\documentclass[a4paper,11pt]{article}
\usepackage{graphicx,epsfig}
\usepackage{fancyhdr,fancybox}
\usepackage{indentfirst}
\usepackage{verbatim}
\usepackage[sort&compress, numbers]{natbib}
\usepackage{geometry}
\usepackage{extarrows,chemarrow,xypic} 
\usepackage{color}
\usepackage[small]{caption2}
\usepackage{microtype}
\DisableLigatures[f]{encoding = *, family = *}

\usepackage{times}     

\usepackage{amssymb}                 
\usepackage{amsthm}
\usepackage{mathrsfs}                
\usepackage{bbm}
\usepackage{hyperref} 

\newcommand{\paperfont}{\fontsize{11pt}{1.2\baselineskip}\selectfont}
\pagebreak[4]
\begin{document}

\theoremstyle{definition}
\makeatletter
\thm@headfont{\bf}
\makeatother
\newtheorem{theorem}{Theorem}[section]
\newtheorem{definition}[theorem]{Definition}
\newtheorem{lemma}[theorem]{Lemma}
\newtheorem{proposition}[theorem]{Proposition}
\newtheorem{corollary}[theorem]{Corollary}
\newtheorem{remark}[theorem]{Remark}
\newtheorem{example}[theorem]{Example}
\newtheorem{assumption}[theorem]{Assumption}

\lhead{}
\rhead{}
\lfoot{}
\rfoot{}

\renewcommand{\refname}{References}
\renewcommand{\figurename}{Figure}
\renewcommand{\tablename}{Table}
\renewcommand{\proofname}{Proof}

\newcommand{\diag}{\mathrm{diag}}
\newcommand{\tr}{\mathrm{tr}}
\newcommand{\re}{\mathrm{Re}}
\newcommand{\one}{\mathbbm{1}}
\newcommand{\loc}{\textrm{loc}}

\newcommand{\Pnum}{\mathbb{P}}
\newcommand{\Enum}{\mathbb{E}}
\newcommand{\Rnum}{\mathbb{R}}
\newcommand{\Cnum}{\mathbb{C}}
\newcommand{\Znum}{\mathbb{Z}}
\newcommand{\Nnum}{\mathbb{N}}
\newcommand{\dnum}{\mathrm{d}}
\newcommand{\abs}[1]{\left\vert#1\right\vert}
\newcommand{\set}[1]{\left\{#1\right\}}
\newcommand{\norm}[1]{\left\Vert#1\right\Vert}
\newcommand{\innp}[1]{\langle {#1}\rangle}
\newcommand{\style}{\setlength{\itemsep}{1pt}\setlength{\parsep}{1pt}\setlength{\parskip}{1pt}}

\title{\textbf{Detailed balance, local detailed balance, and global potential for stochastic chemical reaction networks}}
\author{Chen Jia$^{1}$,\;\;\;Da-Quan Jiang$^{2,3,*}$,\;\;\;Youming Li$^{2}$ \\
\footnotesize $^1$Applied and Computational Mathematics Division, Beijing Computational Science Research Center, Beijing 100193, China \\
\footnotesize $^2$ LMAM, School of Mathematical Sciences, Peking University, Beijing 100871, China. \\
\footnotesize $^3$ Center for Statistical Science, Peking University, Beijing 100871, China. \\
\footnotesize $^{*}$ Correspondence: jiangdq@math.pku.edu.cn}
\date{}
\maketitle
\thispagestyle{empty}

\paperfont

\begin{abstract}	
Detailed balance of a chemical reaction network can be defined in several different ways. Here we investigate the relationship among four types of detailed balance conditions: deterministic, stochastic, local, and zero-order local detailed balance. We show that the four types of detailed balance are equivalent when different reactions lead to different species changes and are not equivalent when some different reactions lead to the same species change. Under the condition of local detailed balance, we further show that the system has a global potential defined over the whole space, which plays a central role in the large deviation theory and the Freidlin-Wentzell-type metastability theory of chemical reaction networks. Finally, we provide a new sufficient condition for stochastic detailed balance, which is applied to construct a class of high-dimensional chemical reaction networks that both satisfies stochastic detailed balance and displays multistability. \\

\end{abstract}

\section{Introduction}
The mathematical theory of chemical reaction networks has attracted massive attention over the past two decades due to its wide applications in biology, chemistry, ecology, and epidemics \cite{anderson2015stochastic}. If a reaction system is well mixed and the numbers of molecules are very large, random fluctuations can be ignored and the evolution of the concentrations of all chemical species can be modeled deterministically as a set of ordinary differential equations (ODEs) based on the law of mass action. If the chemical species are presented in low numbers, however, random fluctuations can no longer be ignored and the evolution of the system is usually modeled stochastically as a continuous-time Markov chain on a high-dimensional lattice, which is widely known as a density-dependent Markov chain \cite{ethier2009markov}. The Kolmogorov forward equation for a density-dependent Markov chain is the well-known chemical master equation, which is first introduced by Delbr\"{u}ck \cite{delbruck1940statistical}. At the center of the mathematical theory of chemical reaction networks is a limit theorem proved by Kurtz \cite{kurtz1971limit, kurtz1972relationship, kurtz1978strong}, which states that when the system size tends to infinity, the trajectories of the stochastic model of a reaction system will converge to those of the deterministic model over any compact time interval, whenever the initial condition converges. Thus far, stochastic reaction networks have served as a fundamental model for the single-cell stochastic gene expression dynamics of gene regulatory networks \cite{peccoud1995markovian, shahrezaei2008analytical, hornos2005self, kumar2014exact, jia2018relaxation, jia2019single, jia2020small, jia2020kinetic}. Recently, the limit theorem of Kurtz has been generalized to stochastic gene regulatory networks with bursting dynamics \cite{jia2017emergent, chen2020limit}.

The limit theorem of Kurtz \cite{kurtz1972relationship} can be viewed as the law of large numbers for stochastic reaction networks. The corresponding large deviation theory and the Freidlin-Wentzell-type metastability theory for stochastic reaction networks have also been studied by many authors \cite{shwartz1995large, li2017large, agazzi2018geometry, agazzi2018seemingly, lazarescu2019large, anderson2020tier} and were rigorously established by Agazzi et al. under the mass action kinetics \cite{agazzi2018large}. At the center of the metastability theory is a quantity called quasi-potential, which plays a crucial role in the analysis of the exit time and exit distribution from a basin of attraction, as well as the most probable transition paths between multiple attractors when the system size is large \cite{olivieri2004large}. However, the quasi-potential is usually defined via an abstract variational expression and hence it is very difficult to obtain a general analytical expression for the quasi-potential.

There are two types of reaction networks that should be distinguished, those satisfy detailed balance and those violate detailed balance. In terms of physical chemistry, detailed balance is a fundamental thermodynamic constraint for closed systems. If there is no sustained energy supply, then a chemical system, when it reaches the steady state, must satisfy detailed balance \cite{qian2007phosphorylation}. In the modelling of many biochemical systems such as enzymes \cite{cornish2012fundamentals} and ion channels \cite{sakmann2009single}, detailed balance has become a basic requirement \cite{alberty2004principle, jia2016solution}. However, in the literature, there are two different definitions of detailed balance for a chemical reaction network. From the deterministic perspective, detailed balance means that there is no net \emph{concentration flux} between any pair of reversible reactions, in which case there is no chemical potential difference and thus the system is in chemical equilibrium. From the stochastic perspective, detailed balance means that there is no net \emph{probability flux} between any pair of microstates on the high-dimensional nonnegative integer lattice, where each microstate is defined as the ordered tuple of concentrations of all chemical species. To distinguish between them, we refer to the former as deterministic detailed balance and the latter as stochastic detailed balance. Some authors believed that the two types of detailed balance are equivalent \cite{anderson2010product}. However, Joshi \cite{joshi2015detailed} pointed out recently that they are not equivalent; they are equivalent when different reactions lead to different species changes, while they are in general not equivalent for systems having two reactions that lead to the same species change --- deterministic detailed balance implies stochastic detailed balance and the opposite is not true.

In this paper, in addition to deterministic and stochastic detailed balance, we propose a third type of detailed balance, which is called \emph{local detailed balance}. This new type of detailed balance characterizes the local asymptotic behavior of a reaction network as the system size tends to infinity. We prove that the three types of detailed balance (deterministic, stochastic, and local) are equivalent when different reactions lead to different species changes, while local detailed balance is even weaker than the other two when some different reactions lead to the same species change --- stochastic detailed balance implies local detailed balance and the opposite is not true. This is the first main result of the present paper. More importantly, under the condition of local detailed balance, we prove that a stochastic reaction network has a global potential that can be computed explicitly and concisely. The global potential reduces to the quasi-potential within each basin of attraction. In general, the quasi-potential is only defined within each basin of attraction. However, local detailed balance guarantees that the system has a global potential that can be defined over the whole space. This is the second main result of the present paper.

In \cite{joshi2015detailed}, Joshi gave the sufficient and necessary condition for deterministic detailed balance. While the author also provided a weaker sufficient condition for stochastic detailed balance, it is difficult to apply it in practice since an infinite number of restrictions need to be verified. In this paper, we provide a simpler sufficient condition for stochastic detailed balance that is more applicable in practice. This new sufficient condition is imposed directly on rate constants and only a finite number of restrictions need to be verified. This is the third main result of the present paper.

The structure of this paper is organized as follows. In Section 2, we recall the basic concepts of chemical reaction networks and state some preliminary results. In Section 3, we reveal the relationship among four types of detailed balance and give some counterexamples. In Section 4, we show that a global potential exists for stochastic reaction networks satisfying local detailed balance and obtain the explicit expression of the global potential. The remaining sections are devoted to the detailed proofs of the main theorems.

\section{Model and preliminary results}
Let $\mathbb{Z}_{\geq 0},\mathbb{R}_{\geq 0}$, and $\mathbb{R}_{>0}$ denote the sets of nonnegative integers, nonnegative real numbers, and positive real numbers, respectively. Recall that a chemical reaction system is composed of a collection of chemical species $\{S_1,\dots,S_d\}$ and a family of reactions
\begin{equation*}
R_i: \sum_{j=1}^d \nu^j_i S_j \xrightarrow{k_i} \sum_{j=1}^d \nu'^j_i S_j,\;\;\;1\leq i\leq N,
\end{equation*}
where $\nu^j_i,\nu'^j_i\in\mathbb{Z}_{\geq 0}$ are the molecule numbers of $S_i$ consumed and created in one instance of that reaction, respectively. For simplicity, we write $\nu_{i} = (\nu_i^1,\dots,\nu_{i}^d)$ and $\nu'_{i}=(\nu'^1_i,\dots,\nu'^d_{i})$, which are called \emph{complexes}. Then the reaction $R_i$ can be written more concisely as $\nu_i \rightarrow \nu'_i$. Moreover, $\nu'_i-\nu_i$ is called the {\em reaction vector} of $R_i$. Let
\begin{equation*}
\mathcal{S}=\{S_1,\dots,S_d\},\;\;\;
\mathcal{C}=\{\nu_i,\nu'_i \ | \ i=1,\dots,N\},\;\;\;
\mathcal{R}=\{R_i | \ i=1,\dots,N\}
\end{equation*}
denote the collections of chemical species, complexes, and reactions respectively. Then the ordered triple $\{\mathcal{S},\mathcal{C},\mathcal{R}\}$ is called a {\em chemical reaction network}.

A chemical reaction network is called {\em reversible} if for any reaction $R_i: \nu_i\rightarrow\nu'_i\in\mathcal{R}$, there exists a reverse reaction $R^-_i:\nu'_i\rightarrow \nu_i\in\mathcal{R}$ \cite{joshi2015detailed}. For any pair of reversible reactions $R_i$ and $R_i^-$, we say that $R_i$ is a {\em forward reaction} if $\nu_i<\nu'_i$, where the symbol ``$<$" is understood in the lexicographic order, namely $\nu_i<\nu'_i$ if and only if $\nu_i^j<\nu'^j_i$ for the first $j$ at which $\nu_i^j$ and $\nu'^j_i$ differ; otherwise, $R_i$ is called a {\em backward reaction}. Throughout the paper, we assume that \emph{all reaction networks under consideration are reversible}.

Most previous papers assumed that different reactions have different reaction vectors. However, in many reaction networks, multiple different reactions may have the same reaction vector. For example, the reaction $S_1\rightarrow S_2$ and the enzyme catalyzed reaction $S_1+E\rightarrow S_2+E$ may coexist in a biochemical reaction system with the latter having a larger reaction rate, where $E$ denotes an enzyme. To cover such systems, here we consider the more general case where each reaction vector may correspond to multiple different reactions. For convenience, we introduce the following definition, which plays an important role in the present paper.

\begin{definition}
	Two reactions are called {\em equivalent} if they have the same reaction vector.
\end{definition}

From this definition, two equivalent reactions are either both forward or both backward. Following the notation in \cite{joshi2015detailed}, let $V(\mathcal{R})=\{\nu'_i-\nu_i|R_i \ \text{is a forward reaction}\}$ denote the collection of reaction vectors for forward reactions. Throughout the paper, the elements in $V(\mathcal{R})$ will be listed as $\omega_1,\cdots,\omega_r$, where $1\leq r\leq N$. For any $\omega_p = (\omega_p^1,\cdots,\omega_p^d)\in V(\mathcal{R})$, we set
\begin{equation*}
\begin{aligned}
\mathcal{R}^+_p&=\{R_i|R_i \ \text{is a forward reaction}, \nu'_i-\nu_i=\omega_p\},\\
\mathcal{R}_p^-&=\{R_i^-|R_i \ \text{is a forward reaction}, \nu'_i-\nu_i=\omega_p\}.
\end{aligned}
\end{equation*}
Then we can relabel the elements in $\mathcal{R}^+_p$ and $\mathcal{R}^-_p$ as
\begin{equation*}
\begin{aligned}
R^+_{pl}:&\sum_{j=1}^d \nu^j_{pl}S_j\xrightarrow{k^+_{pl}} \sum_{j=1}^d \nu'^j_{pl}S_j,\;\;\;1\leq l\leq r_p,\\
R^-_{pl}:&\sum_{j=1}^d \nu'^j_{pl}S_j\xrightarrow{k^-_{pl}} \sum_{j=1}^d \nu^j_{pl}S_j,\;\;\;1\leq l\leq r_p,
\end{aligned}
\end{equation*}
where $r_p$ represents the number of forward reactions with the same reaction vector $\omega_p$ and we call it the \emph{multiplicity} of $\omega_p$. Here we mainly focus on the case of $r_p > 1$ for some $1\leq p\leq r$. For any $1\leq p \leq r$ and $1\leq l \leq r_p$, we set $\nu_{pl}=(\nu^1_{pl},\dots,\nu^d_{pl})$ and $\nu'_{pl}=(\nu'^1_{pl},\dots,\nu'^d_{pl})$.

We first recall the stochastic model of reaction networks. For each $1\leq j\leq d$, let $N_j(t)$ denote the number of molecules of the chemical species $S_j$ at time $t$. Then the concentration of $S_j$ at time $t$ is given by $X^V_j(t) = N_j(t)/V$, where $V$ is the system size. Let $X^V(t)=(X^V_1(t),\dots,X^V_d(t))$ denote the concentration process of all chemical species. At the mesoscopic level, the process $\{X^V(t):t\geq 0\}$ can be modeled by a continuous-time Markov chain on the $d$-dimensional nonnegative integer lattice
\begin{equation*}
E_V = \left\{\frac{n}{V}:n = (n_1,\dots,n_d)\in\mathbb{Z}_{\geq 0}^d\right\}
\end{equation*}
with transition rate matrix $Q^V = (q^V_{x,y})$ whose elements are defined as follows: for any $1\leq p\leq r$ and any $x\in E_V$,
\begin{equation*}
\begin{aligned}
&q^V_{x,x+\frac{\omega_p}{V}}=\sum_{l=1}^{r_p}\frac{k^+_{pl}}{V^{|\nu_{pl}|-1}} \frac{(Vx)!}{(Vx-\nu_{pl})!}, \\
&q^V_{x,x-\frac{\omega_p}{V}}=\sum_{l=1}^{r_p}\frac{k^-_{pl}}{V^{|\nu'_{pl}|-1}} \frac{(Vx)!}{(Vx-\nu'_{pl})!}, \\
&q^V_{x,x}=-\sum_{p=1}^r\left(q^V_{x,x+\frac{\omega_p}{V}}+q^V_{x,x-\frac{\omega_p}{V}}\right),
\end{aligned}
\end{equation*}
where $|\nu|=\sum_{j=1}^d\nu_j$ is called the \emph{order} of the complex $\nu = (\nu_1,\cdots,\nu_d)$ and we write $x!=\prod_{j=1}^d x_j!$ for each vector $x = (x_1,\cdots,x_d)\in\mathbb{Z}^d_{\geq 0}$. Let $\mathbb{P}^V_{x}(t)$ denote the probability of observing state $x\in E_V$ at time $t$. Then the evolution of the stochastic model is governed by the following chemical master equation:
\begin{equation*}
\begin{split}
\frac{{\rm d}\mathbb{P}^V_{x}(t)}{{\rm d}t} =&\;
\sum_{p=1}^rq^V_{x-\frac{\omega_p}{V},x}\mathbb{P}^V_{x-\frac{\omega_p}{V}}(t)+
\sum_{p=1}^rq^V_{x+\frac{\omega_p}{V},x}\mathbb{P}^V_{x+\frac{\omega_p}{V}}(t)\\
&\;-\sum_{p=1}^{r}\left(q^V_{x,x+\frac{\omega_p}{V}}+q^V_{x,x-\frac{\omega_p}{V}}\right)\mathbb{P}^V_{x}(t),\;\;\;
x\in E_V.
\end{split}
\end{equation*}

We next recall the deterministic model of reaction networks. For each $1\leq j\leq d$, let $x_j(t)$ denote the concentration of the chemical species $S_j$ at time $t$. At the macroscopic level, the concentration process $x(t)=(x_1(t),\dots,x_d(t))$ of all chemical species can be modeled by the following ordinary differential equation with mass action kinetics:
\begin{equation}\label{deterministic}
\left\{\begin{aligned}
\frac{{\rm d}x(t)}{{\rm d}t} & =\sum_{p=1}^{r}[f^+_p(x(t))-f^-_p(x(t))]\omega_p, \\
x(0) & =x_0,
\end{aligned}\right.
\end{equation}
where $x_0$ is the initial concentration vector and
\begin{equation}\label{definitionf}
f^+_p(x) = \sum_{l=1}^{r_p}k^+_{pl}x^{\nu_{pl}},\;\;\;
f^-_p(x) = \sum_{l=1}^{r_p}k^-_{pl}x^{\nu'_{pl}},\;\;\;x\in\mathbb{R}^d_{\geq 0},
\end{equation}
where we write $x^y=\prod_{j=1}^d x_j^{y_j}$ for any vectors $x,y\in\mathbb{R}^d_{\geq 0}$. The relationship between the mesoscopic stochastic model and the macroscopic deterministic model is revealed by the following celebrated Kurtz theorem \cite{kurtz1971limit, kurtz1972relationship}: For any $\delta,T>0$, whenever $x_0^V\in E_V$ and $x_0^V\rightarrow x_0\in\mathbb{R}^d_{\geq 0}$, then
\begin{equation}\label{kurtz}
\lim_{V\rightarrow \infty}\mathbb{P}_{x_0^V}(\sup_{t\in [0,T]}\|X^V(t)-x(t)\|\leq\delta)=1,
\end{equation}
where $\mathbb{P}_{x_0^V}(\cdot) = \mathbb{P}(\cdot|X^V(0) = x_0^V)$ and $\|x\|$ denotes Euclidean norm of $x\in \mathbb{R}^d$. This implies that as the system size tends to infinity, the trajectories of the stochastic model will converge to those of the deterministic model on any compact time interval, whenever the initial value converges.

The limit theorem in \eqref{kurtz} can be viewed as the law of large numbers for the stochastic model. The corresponding large deviation principle was proved recently by Agazzi et al. \cite[Theorem 1.6]{agazzi2018large} and is stated as follows. The \emph{Hamiltonian} of a stochastic reaction network is defined as
\begin{equation}\label{hamiltonian}
H(x,\theta) = \sum_{p=1}^{r}\left[f^+_p (x)\left(e^{\omega_p\cdot\theta}-1\right)
+f^-_p (x)\left(e^{-\omega_p\cdot\theta}-1\right)\right],\;\;\;x\in\mathbb{R}^d_{\geq 0},\;\theta\in\mathbb{R}^d,
\end{equation}
where $x\cdot y = \sum_{j=1}^dx_jy_j$ denotes the usual scalar product on $\mathbb{R}^d$. The \emph{Lagrangian} of a stochastic reaction network is then defined as the Legendre-Fenchel transform of the Hamiltonian with respect to the variable $\theta$, namely
\begin{equation}\label{localrate}
L(x,y) = \sup_{\theta\in \mathbb{R}^d}(\theta\cdot y-H(x,\theta)),\;\;\;
x\in\mathbb{R}^d_{\geq 0},\;y\in\mathbb{R}^d.
\end{equation}
The Lagrangian is nonnegative because $L(x,y)\geq 0\cdot y-H(x,0) = 0$. Moreover, it is not hard to prove that $L(x,y) = \infty$ for any $y\notin\mathrm{span}(V(\mathcal{R}))$. This is because any $y\notin\mathrm{span}(V(\mathcal{R}))$ can be decomposed uniquely as $y = y_1+y_2$, where $y_1\in\mathrm{span}(V(\mathcal{R}))$ and $0\neq y_2\in\mathrm{span}(V(\mathcal{R}))^\perp$. Thus for any $K>0$, we have $L(x,y)\geq Ky_2\cdot y-H(x,Ky_2) = K\|y_2\|^2$, where we have used the fact that $H(x,\theta) = 0$ for any $\theta\in\mathrm{span}(V(\mathcal{R}))^\perp$. Since $K$ is arbitrarily chosen, we conclude that $L(x,y) = \infty$.

To proceed, let $D_{[0,T]}(\mathbb{R}^d_{\geq 0})$ denote the space of c\`{a}dl\`{a}g functions $\phi:[0,T]\rightarrow \mathbb{R}^d_{\geq 0}$ equipped with the topology of uniform convergence. For any $x_0\in\mathbb{R}^d_{\geq 0}$, let $I_{x_0,T}:D_{[0,T]}(\mathbb{R}^d_{\geq 0})\rightarrow[0,\infty]$ be the function defined as
\begin{equation*}
I_{x_0,T}(\phi) =
\begin{cases}
\int_0^T L(\phi(t),\dot{\phi}(t)){\rm d}t,\;\;\;&\text{if $\phi$ is absolutely continuous and $\phi(0) = x_0$},\\
\infty,\;\;\;&\text{otherwise}.
\end{cases}
\end{equation*}
Using the properties of the Lagrangian, it is easy to see that $I_{x_0,T}(\phi) = \infty$ if there exists $0\leq t\leq T$ such that $\phi(t)\notin x_0+\mathrm{span}(V(\mathcal{R}))$. With these notation, Agazzi et al. \cite[Theorem 1.6]{agazzi2018large} proved the following result: provided that the network is asiphonic and strongly endotactic (see \cite[Definitions 1.8 and 1.9]{agazzi2018large} for detailed definitions), for any $x^V_0\in E_V$ and $x^V_0\rightarrow x_0$, the law of the process $\{X^V(t):t\in [0,T]\}$ with $X^V(0) = x^V_0$ satisfies a large deviation principle with rate $V$ and good \emph{rate function} $I_{x_0,T}$. The large deviation principle means that for any measurable set $\Gamma\subset D_{[0,T]}(\mathbb{R}^d_{\geq 0})$, we have
\begin{equation}\label{largedeviation}
\begin{split}
\liminf_{V\rightarrow\infty}\frac{1}{V}\log\mathbb{P}_{x_0^V}(X^V(\cdot)\in\Gamma^o) &\geq -\inf_{\phi\in\Gamma^o}I_{x_0,T}(\phi),\\
\limsup_{V\rightarrow\infty}\frac{1}{V}\log\mathbb{P}_{x_0^V}(X^V(\cdot)\in\bar\Gamma) &\leq -\inf_{\phi\in\bar\Gamma}I_{x_0,T}(\phi),
\end{split}
\end{equation}
where $\Gamma^o$ and $\bar\Gamma$ denote the interior and closure of $\Gamma$, respectively. Combining \eqref{kurtz} and \eqref{largedeviation}, it is easy to see that $I_{x_0,T}(x)=0$, where $x = x(t)$ is the solution of the deterministic model \eqref{deterministic}. The rate function can be used to define the following {\em quasi-potential}:
\begin{equation}\label{quasipotential}
W(x_0,y) = \inf{\{I_{x_0,T}(\phi):\phi(0)=x_0,\;\phi(T)=y,\;T\geq 0\}},\;\;\;x_0,y\in\mathbb{R}^d_{\geq 0}.
\end{equation}
Intuitively, $W(x_0,y)$ represents the ``cost" for the stochastic reaction network to move from $x_0$ to $y$. It is easy to see that the quasi-potential is nonnegative and jointly continuous in $x_0$ and $y$ \cite{olivieri2004large}. Using the properties of the Lagrangian, it is easy to see that $W(x_0,y) = \infty$ if $y\notin x_0+\mathrm{span}(V(\mathcal{R}))$.

Agazzi et al. \cite[Theorem 1.15]{agazzi2018large} also deals with the Freidlin-Wentzell-type metastability theory for chemical reaction networks, where the quasi-potential plays a central role. For simplicity, we consider the case where the domain under consideration contains only one stable equilibrium point. Specifically, we assume that the following four conditions are satisfied:\\
(a) Let $D$ be a bounded open domain in $\mathbb{R}^d_{\geq 0}$ with a piecewise $C^2$ boundary $\partial D$.\\
(b) Let $c\in D$ be an asymptotically stable equilibrium point of the deterministic model \eqref{deterministic}.\\
(c) The set $\bar{D} = D\cup\partial D$ is attracted to $c$, which means that whenever $x_0\in\bar D$, the solution of the deterministic model \eqref{deterministic} starting from $x_0$ satisfies $x(t)\in D$ for each $t>0$ and $x(t)\rightarrow c$ as $t\rightarrow\infty$.\\
(d) There exists a ball $B\subset\bar D$ such that for any $x\in B$ and $y\in\bar D$, the set $\bar D$ contains the line segment between $x$ and $y$.

It is easy to check that Assumptions A.3 and A.4 in \cite{agazzi2018large} are satisfied under these conditions. Then the Kurtz theorem implies that when $V$ is sufficiently large, the trajectory of the stochastic model will stay in the domain $D$ over any compact time interval with overwhelming probability. However, it is still possible for the system to escape from $D$. The mean exit time from $D$ has the following asymptotic behavior:
\begin{equation*}
\lim_{V\rightarrow\infty}\frac{1}{V}\log\mathbb{E}_{x_0^V}\tau_V = \inf_{y\in \partial D}W(c,y),
\end{equation*}
where $\tau_V=\inf\{t\geq 0:X^V(t)\notin D\}$ denotes the exit time of $X^V$ from $D$. Moreover, if there is a unique $y_0\in\partial D$ such that
\begin{equation*}
W(c,y_0) = \inf_{y\in \partial D}W(c,y),
\end{equation*}
then for any $\delta>0$, the exit position from $D$ has the following asymptotic behavior:
\begin{equation*}
\lim_{V\rightarrow\infty}\mathbb{P}_{x_0^V}(\|X^V(\tau_V)-y_0\|<\delta) = 1,
\end{equation*}
and for any $\delta>0$ and $z_0\in\partial D$, the exit distribution from $D$ has the following asymptotic behavior:
\begin{equation*}
\lim_{\delta\rightarrow 0}\lim_{V\rightarrow\infty}
\frac{1}{V}\log\mathbb{P}_{x_0^V}(\|X^V(\tau_V)-z_0\|<\delta) = W(c,y_0)-W(c,z_0).
\end{equation*}
Intuitively, when $V$ is sufficiently large, the stochastic model will escape from $D$ around a particular point $y_0\in\partial D$ at which the quasi-potential restricted to $\partial D$ attains its minimum.

\section{Detailed balance for chemical reaction networks}
In this section, we investigate the relationship among different types of detailed balance conditions for chemical reaction networks. Before stating our results, we first recall the definitions of deterministic and stochastic detailed balance for chemical reaction networks \cite{joshi2015detailed}.

\begin{definition}
	We say that a reaction network satisfies \emph{deterministic detailed balance} (or reaction network detailed balance \cite{joshi2015detailed}) if there exists $c\in\mathbb{R}^d_{>0}$ such that
	\begin{equation}\label{detailed}
	k^+_{pl}c^{\nu_{pl}}=k^-_{pl}c^{\nu'_{pl}},\;\;\;\text{for any\;}1\leq p\leq r,\;1\leq l\leq r_p.
	\end{equation}
	Here $c$ is called a \emph{chemical equilibrium state} of the reaction network.
\end{definition}

Clearly, any chemical equilibrium state $c$ is also an equilibrium point of the deterministic model \eqref{deterministic} and thus it is also called a \emph{detailed balanced equilibrium point}. It has been shown that for mass action kinetics, if one positive equilibrium point of the deterministic model is detailed balanced, then every positive equilibrium point is detailed balanced \cite{joshi2015detailed,1989Necessary}.

\begin{definition}
	We say that a reaction network satisfies \emph{stochastic detailed balance} (or Markov chain detailed balance \cite{joshi2015detailed}) if for any $V>0$, there exists a probability measure $\pi^V = (\pi^V_x)$ on $E_V$ such that
	\begin{equation*}
	\pi^V_{x}q^V_{x,y}=\pi^V_{y}q^V_{y,x},\;\;\;\text{for any\;}x,y\in E_V.
	\end{equation*}
	Note that here we do not require that $\pi^V$ is a probability distribution.
\end{definition}

A simple method of verifying stochastic detailed balance is to use the \emph{Kolmogorov criterion} \cite{kolmogoroff1936theorie}, which states that a reaction network satisfies stochastic detailed balance if and only if for any $V>0$, the transition rates satisfy the following \emph{Kolmogorov cycle condition}:
\begin{equation*}
q^V_{x_1,x_2}q^V_{x_2,x_3}\cdots q^V_{x_n,x_1} = q^V_{x_2,x_1}q^V_{x_3,x_2}\cdots q^V_{x_1,x_n}
\end{equation*}
for any finite number of states $x_1,\dots,x_n\in E_V$. In other words, the Kolmogorov criterion states that a reaction network satisfies stochastic detailed balance if and only if for any $V>0$, the product of the transition rates of the stochastic model along any cycle is equal to that along the reversed cycle.

Besides deterministic and stochastic detailed balance, we introduce another type of detailed balance which is defined as follows. This new type of detailed balance will play an important role in constructing the global potential of a chemical reaction network.

\begin{definition}\label{conditions}
	(i) We say that a reaction network satisfies {\em zero-order local detailed balance} if for any integers $\xi_1,\dots,\xi_r$ satisfying
	$\sum_{p=1}^r \xi_p\omega_p = 0$, we have
	\begin{equation}\label{continuouscyclecondition}
	\sum_{p=1}^r \xi_p\log \frac{f^+_p(x)}{f^-_p(x)}=0,\;\;\;\text{for any\;}x\in \mathbb{R}_{>0}^d,
	\end{equation}
	where $f^+_p(x)$ and $f^-_p(x)$ are the functions defined in \eqref{definitionf}. \\
	(ii) We say that a reaction network satisfies {\em first-order local detailed balance} if for any $1\leq p,q\leq r$ with $p\neq q$, we have
	\begin{equation}\label{cparallelogram}
	\omega_{q}\cdot\nabla\left(\log \frac{f^+_{p}(x)}{f^-_{p}(x)}\right)
	= \omega_{p}\cdot\nabla\left(\log \frac{f^+_{q}(x)}{f^-_{q}(x)}\right),\;\;\;
	\text{for any\;}x\in \mathbb{R}_{>0}^d.
	\end{equation}
	(iii) We say that a reaction network satisfies {\em local detailed balance} if it satisfies both zero-order and first-order local detailed balance.
\end{definition}

\begin{remark}
	The ideas behind the above definition are explained as follows. For any integers $\xi_1,\dots,\xi_r$ satisfying $\sum_{p=1}^r\xi_p\omega_p = 0$, we can construct a cycle $C$ in the integer lattice $\mathbb{Z}^d$, which is given by
	\begin{equation*}
	\begin{split}
	&C: 0\rightarrow\mathrm{sgn}(\xi_1)\omega_1\rightarrow\cdots\rightarrow \xi_1\omega_1\\ &\rightarrow\xi_1\omega_1+\mathrm{sgn}(\xi_2)\omega_2\rightarrow\cdots\rightarrow
	\xi_1\omega_1+\xi_2\omega_2\rightarrow\cdots \\
	&\rightarrow\xi_1\omega_1+\xi_2\omega_2+\cdots+\mathrm{sgn}(\xi_r)\omega_r\rightarrow\cdots\rightarrow \xi_1\omega_1+\xi_2\omega_2+\cdots+\xi_r\omega_r = 0,
	\end{split}
	\end{equation*}
	where $\mathrm{sgn}(x)$ is the sign function which takes the value of $1$ if $x>0$, takes the value of $0$ if $x=0$, and takes the value of $-1$ if $x<0$. Obviously, for any $x^V\in E_V$ and $x^V\rightarrow x\in\mathbb{R}^d_{>0}$, the cycle $C$ can induce a cycle $x^V+C/V$ in $E_V$ around $x$, which becomes increasingly smaller as $V$ increases. For convenience, let $\eta_1\rightarrow\eta_2\rightarrow\dots\rightarrow \eta_L\rightarrow\eta_1$ denote the induced cycle in $E_V$, where $L = \sum_{p=1}^r|\xi_p|$ is the number of transitions in the cycle. If a reaction network satisfies stochastic detailed balance, then it follows from Kolmogorov's cycle condition that
		\begin{equation}\label{kolmogorov}
		f\left(\frac{1}{V}\right) := \log \frac{q^V_{\eta_1,\eta_2}q^V_{\eta_2,\eta_3}\cdots q^V_{\eta_L,\eta_1}}{q^V_{\eta_2,\eta_1}q^V_{\eta_3,\eta_2}\cdots q^V_{\eta_1,\eta_L}}=0.
		\end{equation}
		Note that the left-hand side of this equality is a function of $1/V$. Since $f(1/V) = 0$ for all $V>0$, we have
		\begin{equation*}
		f(0) := \lim_{V\rightarrow\infty}f\left(\frac{1}{V}\right) = 0\;\;\;\textrm{(zero-order information)},
		\end{equation*}
		and
		\begin{equation*}
		f'(0) := \lim_{V\rightarrow\infty}\frac{f\left(\frac{1}{V}\right)-f(0)}{\frac{1}{V}} = 0\;\;\;\textrm{(first-order information)}.
		\end{equation*}
		Roughly speaking, the condition \eqref{continuouscyclecondition} extracts the zero-order information of the equality \eqref{kolmogorov} as $V\rightarrow\infty$ and the condition \eqref{cparallelogram} extracts the first-order information of the equality \eqref{kolmogorov} as $V\rightarrow\infty$. Since the induced cycle becomes increasingly smaller as $V$ increases, \eqref{continuouscyclecondition} and  \eqref{cparallelogram} actually contain the zero-order and first-order local information of detailed balance around $x\in\mathbb{R}^d_{>0}$, respectively.
\end{remark}

It is a well-known result that deterministic detailed balance implies stochastic detailed balance for a chemical reaction network \cite[Theorem 5.9]{joshi2015detailed}. The following theorem reveals the relationship between stochastic and local detailed balance.

\begin{theorem}\label{result1}
	If a reaction network satisfies stochastic detailed balance, then it also satisfies local detailed balance. In other words, stochastic detailed balance implies local detailed balance.
\end{theorem}

\begin{proof}
	The proof of the theorem will be given in Section \ref{result1proof}.
\end{proof}

The next corollary follows from Theorem \ref{result1} and \cite[Theorem 5.9]{joshi2015detailed} immediately.

\begin{corollary}\label{relationship}
	For a chemical reaction network, the following statements hold:\\
	(a) Deterministic detailed balance implies stochastic detailed balance.\\
	(b) Stochastic detailed balance implies local detailed balance.\\
	(c) Local detailed balance implies zero-order local detailed balance.
\end{corollary}

The above corollary reveals the inclusion relationship among four types of detailed balance: deterministic, stochastic, local, and zero-order local detailed balance, as illustrated in Fig. \ref{containment}. Deterministic detailed balance is the strongest and zero-order local detailed balance is the weakest. The following proposition reveals when the four types of detailed balance are equivalent.
\begin{figure}[htb]
	\centering\includegraphics[width=100mm]{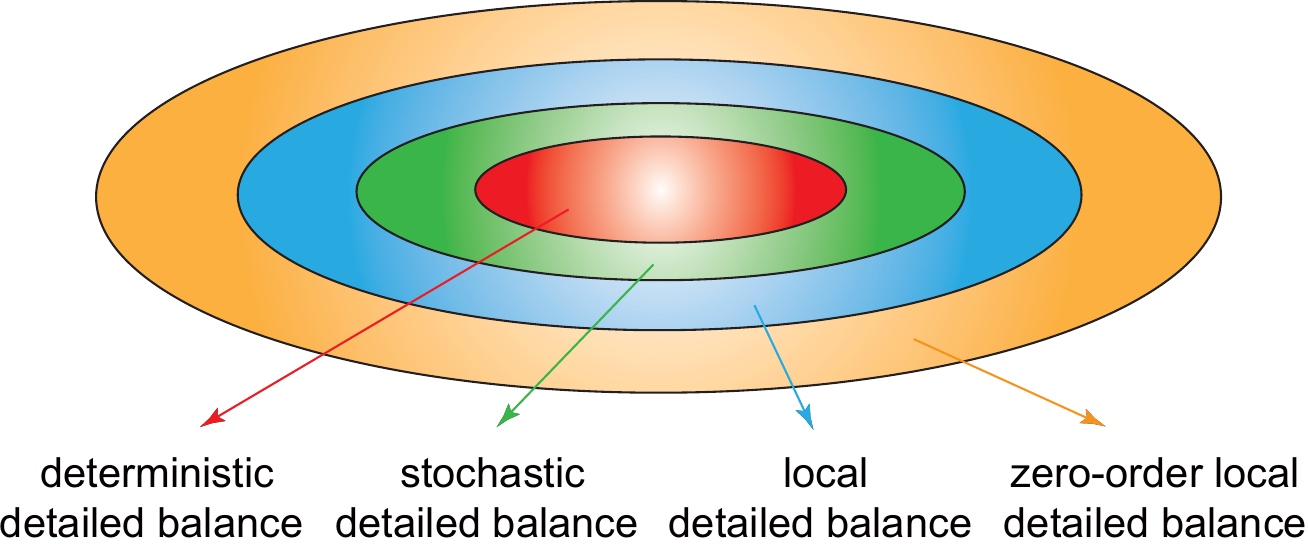}
	\caption{Inclusion relationship among four types of detailed balance conditions for chemical reaction networks: deterministic, stochastic, local, and zero-order local detailed balance. The four conditions are equivalent for chemical networks without equivalent reactions and are not equivalent for chemical networks with equivalent reactions.}
	\label{containment}
\end{figure}

\begin{proposition}
	If a chemical network has no equivalent reactions, then the following statements are equivalent:\\
	(a) The network satisfies deterministic detailed balance.\\
	(b) The network satisfies stochastic detailed balance.\\
	(c) The network satisfies local detailed balance.\\
	(d) The network satisfies zero-order local detailed balance.
\end{proposition}

\begin{proof}
	By Corollary \ref{relationship}, we only need to prove that (d) implies (a). If the network satisfies zero-order local detailed balance, for any integers $\xi_1,\cdots,\xi_r$ satisfying $\sum_{p=1}^r\xi_p\omega_p=0$, we have
	\begin{equation*}
	\sum_{p=1}^r\xi_p\log\frac{f^+_p(x)}{f^-_p(x)} = 0,\;\;\;x\in\mathbb{R}^d_{>0}.
	\end{equation*}
	Since the network has no equivalent reactions, we have $r_p=1$ for any $1\leq p\leq r$. This shows that
	\begin{equation}\label{noequivalent}
	\log\frac{f^+_p(x)}{f^-_p(x)} = \log\left(\frac{k^+_{p1}}{k^-_{p1}}x^{-\omega_p}\right)
	= \log\frac{k^+_{p1}}{k^-_{p1}}-\omega_p\cdot\log x ,
	\end{equation}
	where $\log x = (\log x_1,\cdots,\log x_d)$. Combining the above two equations shows that
	\begin{equation*}
	0 = \sum_{p=1}^r \xi_p\log\frac{f^+_p(x)}{f^-_p(x)}
	=\sum_{p=1}^r \xi_p\log\frac{k^+_{p1}}{k^-_{p1}}-\sum_{p=1}^r \xi_p\omega_p\cdot \log x
	=\sum_{p=1}^r \xi_p\log\frac{k^+_{p1}}{k^-_{p1}} = 0.
	\end{equation*}
	Thus we conclude that for any integers $\xi_1,\cdots,\xi_r$ satisfying $\sum_{p=1}^r\xi_p(\nu_{p1}'-\nu_{p1}) = 0$, we have
	\begin{equation*}
	\sum_{p=1}^r \xi_p\log\frac{k^+_{p1}}{k^-_{p1}} = 0.
	\end{equation*}
	This is exactly the so-called Wegscheider cycle condition, which is widely known as the sufficient and necessary condition for deterministic detailed balance \cite[Proposition 1]{2011Extended}. Therefore, we have proved that (d) implies (a).
\end{proof}

We have seen that if a chemical network has no equivalent reactions, then the four types of detailed balance are equivalent. For chemical networks having equivalent reactions, however, the four types of detailed balance are no longer equivalent, as can be seen from the following three counterexamples.

The following example \cite{vellela2008stochastic, joshi2015detailed} gives a reaction network that satisfies stochastic detailed balance but violates deterministic detailed balance.

\begin{example}\label{schlogl}
	Consider the following well-known Schl{\"o}gl model \cite{vellela2008stochastic}:
	\begin{equation*}
	\varnothing\autorightleftharpoons{$k^+_{11}$}{$k^-_{11}$} S_1, \;\;\;  2S_1 \autorightleftharpoons{$k^+_{12}$}{$k^-_{12}$} 3S_1.
	\end{equation*}
	The stochastic model of this reaction network is a one-dimensional birth-death process and thus must satisfy stochastic detailed balance. Moreover, it is easy to check that deterministic detailed balance is satisfied if and only if $k^+_{11}/k^-_{11}=k^+_{12}/k^-_{12}$ \cite{vellela2008stochastic}. In other words, if $k^+_{11}/k^-_{11}\neq k^+_{12}/k^-_{12}$, then deterministic detailed balance is violated.
\end{example}

The following example gives a reaction network that satisfies local detailed balance but violates stochastic detailed balance.

\begin{example}
	Consider the following chemical reaction system:
	\begin{gather*}
	\varnothing\autorightleftharpoons{$k^+_1$}{$k^-_{1}$} S_1,\;\;\; S_1\autorightleftharpoons{$k^+_2$}{$k^-_{2}$} 2S_1,\\
	\varnothing\autorightleftharpoons{$(k^+_1)^2$}{$(k^-_{1})^2$} 2S_1,\;\;\; S_1\autorightleftharpoons{$2k_1^+k^+_2$}{$2k^-_{1}k^-_2$} 3S_1, \;\;\;  2S_1\autorightleftharpoons{$(k^+_2)^2$}{$(k^-_{2})^2$} 4S_1.
	\end{gather*}
	By definition, the forward reactions are given by
	\begin{equation*}
	\varnothing\xrightarrow{k^+_{1}}S_1,\;\;\;
	S_1\xrightarrow{k^+_{2}}2S_1,\;\;\;
	\varnothing\xrightarrow{(k^+_{1})^2}2S_1,\;\;\;
	S_1\xrightarrow{2k^+_{1}k_2^+}3S_1,\;\;\;
	2S_1\xrightarrow{(k^+_2)^2}4S_1,
	\end{equation*}
	and the backward reactions are given by
	\begin{equation*}
	S_1\xrightarrow{k^-_{1}}\varnothing,\;\;\;
	2S_1\xrightarrow{k^-_{2}} S_1, \;\;\;
	2S_1\xrightarrow{(k^-_{1})^2} \varnothing,\;\;\;
	3S_1\xrightarrow{2k^-_{1}k^-_2} S_1,\;\;\;
	4S_1\xrightarrow{(k^-_2)^2} 2S_1.
	\end{equation*}
	It is easy to see that the first two forward reactions have the same reaction vector $\omega_1=1$ and the last three forward reactions also have the same reaction vector $\omega_2=2$. The multiplicities of the two reaction vectors are given by $r_1=2$ and $r_2=3$, respectively.
	
	We first prove that the system satisfies local detailed balance. Clearly, the two reaction vectors are linearly related by $\xi_1\omega_1+\xi_2\omega_2 = 0$ with $\xi_1 = 2$ and $\xi_2 = -1$. It is easy to check that
	\begin{equation*}
	\begin{split}
	\log\frac{f_2^+(x)}{f_2^-(x)}
	&= \log\frac{(k^+_1)^2+2k^+_1k^+_2x+(k^+_2)^2x^2}{(k^-_1)^2x^2+2k^-_1k^-_2x^3+(k^-_2)^2x^4}
	= \log\frac{(k^+_1+k^+_2x)^2}{(k^-_1x+k^-_2x^2)^2}\\
	&= 2\log\frac{k^+_1+k^+_2x}{k^-_1x+k^-_2x^2} = 2\log\frac{f_1^+(x)}{f_1^-(x)}.
	\end{split}
	\end{equation*}
	Therefore, we have
	\begin{equation*}
	\xi_1\log\frac{f_1^+(x)}{f_1^-(x)}+\xi_2\log\frac{f_2^+(x)}{f_2^-(x)}
	= 2\log\frac{f_1^+(x)}{f_1^-(x)}-\log\frac{f_2^+(x)}{f_2^-(x)} = 0,
	\end{equation*}
	which shows that zero-order local detailed balance is satisfied. Moreover, we have
	\begin{equation*}
	\omega_2\cdot\nabla\left(\log\frac{f^+_1(x)}{f^-_1(x)}\right)
	-\omega_1\cdot\nabla\left(\log\frac{f^+_2(x)}{f^-_2(x)}\right)
	= 2\frac{{\rm d}}{{\rm d} x}\left(\log\frac{f^+_1(x)}{f^-_1(x)}\right)
	-\frac{{\rm d}}{{\rm d} x}\left(\log\frac{f^+_2(x)}{f^-_2(x)}\right) = 0,
	\end{equation*}
	which shows that first-order local detailed balance is also satisfied.
	
	We next prove that the system violates stochastic detailed balance. For each $V>0$, consider the following cycle in $E_V$:
	\begin{equation*}
	\frac{n}{V}\rightarrow \frac{n+1}{V}\rightarrow \frac{n+2}{V}\rightarrow \frac{n}{V}.
	\end{equation*}
	The transition rates along this cycle and its reversed cycle are given by
	\begin{gather*}
	q^V_{\frac{n}{V},\frac{n+1}{V}} = k^+_1V+k_2^+n,\;\;\;
	q^V_{\frac{n+1}{V},\frac{n}{V}} = k^-_1(n+1)+\frac{k^-_2}{V}(n+1)n,\\
	q^V_{\frac{n+1}{V},\frac{n+2}{V}} = k^+_1V+k_2^+(n+1),\;\;\;
	q^V_{\frac{n+2}{V},\frac{n+1}{V}} = k^-_1(n+2)+\frac{k^-_2}{V}(n+2)(n+1),\\
	q^V_{\frac{n+2}{V},\frac{n}{V}}
	= \frac{(k^-_1)^2(n+2)(n+1)}{V}+\frac{2k^-_1k^-_2(n+2)(n+1)n}{V^2}+\frac{(k^-_2)^2(n+2)(n+1)n(n-1)}{V^3},\\
	q^V_{\frac{n}{V},\frac{n+2}{V}} = (k_1^+)^2V+2k_1^+k_2^+n+\frac{(k_2^+)^2}{V}n(n-1).
	\end{gather*}
	Direct computation shows that
	\begin{equation*}
	\begin{split}
	&\;\frac{q^V_{\frac{n}{V},\frac{n+1}{V}}q^V_{\frac{n+1}{V},\frac{n+2}{V}}q^V_{\frac{n+2}{V},\frac{n}{V}}}
	{q^V_{\frac{n+1}{V},\frac{n}{V}}q^V_{\frac{n+2}{V},\frac{n+1}{V}}q^V_{\frac{n}{V},\frac{n+2}{V}}}\\
	=&\;\frac{(k^+_1V+k_2^+n)[k^+_1V+k_2^+(n+1)][(k^-_1)^2V^2+2k^-_1k^-_2nV+(k^-_2)^2n(n-1)]}
	{(k^-_1V+k^-_2n)[k^-_1V+k^-_2(n+1)][(k_1^+)^2V^2+2k^+_1k^+_2nV+(k_2^+)^2n(n-1)]}.
	\end{split}
	\end{equation*}
	It is easy to check that the left-hand side of this equation is equal to 1 if and only if $k_1^+/k^-_1=k^+_2/k^-_2$, which means that stochastic detailed balance is violated if $k_1^+/k^-_1\neq k^+_2/k^-_2$.
\end{example}

The following example gives a reaction network that satisfies zero-order local detailed balance but violates local detailed balance.

\begin{example}
	Consider the following chemical reaction system:
	\begin{equation*}
	\varnothing\autorightleftharpoons{$k^+_1$}{$k^-_{1}$} S_1,\;\;\;
	S_2 \autorightleftharpoons{$k^+_2$}{$k^-_{2}$} S_1+S_2,\;\;\;
	\varnothing \autorightleftharpoons{$k^-_1$}{$k^+_{1}$} S_2,\;\;\;
	S_2 \autorightleftharpoons{$k^-_2$}{$k^+_{2}$} 2S_2,\;\;\;
	\varnothing \autorightleftharpoons{$1$}{$1$} S_1+S_2.
	\end{equation*}
	By definition, the forward reactions are given by
	\begin{equation*}
	\varnothing\xrightarrow{k^+_{1}}S_1,\;\;\;
	S_2\xrightarrow{k^+_{2}}S_1+S_2, \;\;\;
	\varnothing\xrightarrow{k^-_{1}}S_2, \;\;\;
	S_2\xrightarrow{k^-_{2}}2S_2,\;\;\;
	\varnothing\xrightarrow{1}S_1+S_2,
	\end{equation*}
	and the backward reactions are given by
	\begin{equation*}
	S_1\xrightarrow{k^-_{1}}\varnothing,\;\;\;
	S_1+S_2\xrightarrow{k^-_{2}}S_2, \;\;\;
	S_2\xrightarrow{k^+_{1}}\varnothing, \;\;\;
	2S_2\xrightarrow{k^+_{2}}S_2,\;\;\;
	S_1+S_2\xrightarrow{1}\varnothing.
	\end{equation*}
	The first two forward reactions have the same reaction vector $\omega_1=(1,0)$, the next two forward reactions have the same reaction vector $\omega_2=(0,1)$, and the reaction vector of the last forward reaction is given by $\omega_3=(1,1)$. The multiplicities of the three reaction vectors are given by $r_1 = r_2=2$ and $r_3=1$, respectively.
	
	We first prove that the system satisfies zero-order local detailed balance. Clearly, the three reaction vectors are linearly related by $\xi_1\omega_1+\xi_2\omega_2+\xi_3\omega_3 = 0$ with $\xi_1 = \xi_2= 1$ and $\xi_3 = -1$. It is easy to check that
	\begin{gather*}
	\frac{f_1^+(x)}{f_1^-(x)} = \frac{k^+_1+k^+_2x_2}{k^-_1x_1+k^-_2x_1x_2},\;\;\;
	\frac{f_2^+(x)}{f_2^-(x)} = \frac{k^-_1+k^-_2x_2}{k^+_1x_2+k^+_2x_2^2},\;\;\;
	\frac{f_3^+(x)}{f_3^-(x)} = \frac{1}{x_1x_2}.
	\end{gather*}
	Therefore, we have
	\begin{equation*}
	\xi_1\log\frac{f_1^+(x)}{f_1^-(x)}+\xi_2\log\frac{f_2^+(x)}{f_2^-(x)}+\xi_3\log\frac{f_3^+(x)}{f_3^-(x)}
	= \log\frac{f_1^+(x)}{f_1^-(x)}+\log\frac{f_2^+(x)}{f_2^-(x)}-\log\frac{f_3^+(x)}{f_3^-(x)} = 0,
	\end{equation*}
	which shows that zero-order local detailed balance is satisfied. On the other hand, it is easy to check that
	\begin{gather*}
	\omega_2\cdot\nabla\left(\log\frac{f^+_1(x)}{f^-_1(x)}\right)
	= \frac{\partial}{\partial x_2}\left(\log \frac{k_1^++k_2^+x_2}{k^-_1x_1+k^-_2x_1x_2}\right)
	= \frac{k^-_1k^+_2-k^+_1k^-_2}{(k^+_1+k^+_2x_2)(k^-_1+k^-_2x_2)},\\
	\omega_1\cdot\nabla\left(\log\frac{f^+_2(x)}{f^-_2(x)}\right)
	= \frac{\partial}{\partial x_1}\left(\log \frac{k^-_1+k^-_2x_2}{k_1^+x_2+k_2^+x^2_2}\right)
	= 0.
	\end{gather*}
	Clearly, the left-hand sides of the above two equations are equal if and only if $k_1^+/k^-_1=k^+_2/k^-_2$. In other words, if $k_1^+/k^-_1\neq k^+_2/k^-_2$, then first-order local detailed balance is violated and thus the system does not satisfy local detailed balance.
\end{example}

In \cite[Theorem 4.2]{joshi2015detailed}, Joshi gave the following sufficient and necessary condition for deterministic detailed balance.

\begin{theorem}\label{joshi}\cite[Theorem 4.2]{joshi2015detailed}
	A reaction network satisfies deterministic detailed balance if and only if the following two conditions are satisfied:\\
	(a) For any $1\leq p\leq r$, the rate constants of the reactions in $\mathcal{R}^+_p$ and $\mathcal{R}^-_p$ satisfy
	\begin{equation*}\
	\frac{k^+_{p1}}{k^-_{p1}} = \frac{k^+_{p2}}{k^-_{p2}} = \cdots = \frac{k^+_{pr_p}}{k^-_{pr_p}}.
	\end{equation*}
	(b) For any integers $\xi_1,\cdots,\xi_r$ satisfying $\sum_{p=1}^r\xi_p\omega_p=0$, we have
	\begin{equation}\label{joshia}
	\sum_{p=1}^r\xi_p\log\frac{k^+_{p1}}{k^-_{p1}} = 0.
	\end{equation}
\end{theorem}

While Joshi \cite[Theorem 5.14]{joshi2015detailed} also gave a sufficient condition for stochastic detailed balance, it is difficult to apply it in practice since an infinite number of restrictions need to be verified. Here we give a simpler sufficient condition for stochastic detailed balance that is more applicable in practice. To state our sufficient condition, we need the following definition.

\begin{definition}\label{definition}
	For each $1\leq p\leq r$, we say that the reaction vector $\omega_p$ satisfies the {\em orthogonality condition} if
	\begin{equation*}
	(\nu^j_{pl_1}-\nu^j_{pl_2})\omega_{q}^j=0,\;\;\;
	\text{for any\;}1\leq l_1,l_2\leq r_p,\;q\neq p,\;1\leq j\leq d.
	\end{equation*}
\end{definition}

It is easy to see that if $r_p = 1$, then the orthogonality condition is automatically satisfied for $\omega_p$. The following theorem provides a new sufficient condition for stochastic detailed balance. This sufficient condition is imposed directly on rate constants and only a finite number of restrictions need to be verified.

\begin{theorem}\label{result4}
	Suppose that the reaction vectors $\omega_1,\cdots,\omega_r$ are linearly independent. Suppose that for each $1\leq p\leq r$, either one of the following two conditions is satisfied:\\
	(a) The reaction vector $\omega_p$ satisfies the orthogonality condition. \\
	(b) The rate constants of the reactions in $\mathcal{R}^+_p$ and $\mathcal{R}^-_p$ satisfy
	\begin{equation}\label{joshib}
	\frac{k^+_{p1}}{k^-_{p1}}=\frac{k^+_{p2}}{k^-_{p2}}=\cdots=\frac{k^+_{pr_p}}{k^-_{pr_p}}.
	\end{equation}\\
	Then the reaction network satisfies stochastic detailed balance.
\end{theorem}

\begin{proof}
	The proof of this theorem will be given in Section \ref{result4proof}.
\end{proof}

From Theorem \ref{joshi}, if the reaction vectors $\omega_1,\cdots,\omega_r$ are linearly independent, then \eqref{joshia} holds trivially and hence the reaction network satisfies deterministic detailed balance if and only if \eqref{joshib} is satisfied for all $1\leq p\leq r$.

We next use Theorem \ref{result4} to construct more examples of chemical reaction networks that satisfy stochastic detailed balance but violate deterministic detailed balance.

\begin{example}\label{bistable}
	Consider the following chemical reaction system:
	\begin{equation*}
	\varnothing \autorightleftharpoons{$k^+_{11}$}{$k^-_{11}$}S_1,\;\;\;
	S_1\autorightleftharpoons{$k^+_{12}$}{$k^-_{12}$} 2S_1,\;\;\; S_2\autorightleftharpoons{$k^-_{21}$}{$k^+_{21}$} S_1,\;\;\;
	3S_2\autorightleftharpoons{$k^-_{22}$}{$k^+_{22}$} S_1+2S_2.
	\end{equation*}
	By definition, the forward reactions are given by
	\begin{equation*}
	\varnothing\xrightarrow{k^+_{11}}S_1,\;\;\;
	S_1\xrightarrow{k^+_{12}}2S_1,\;\;\;
	S_2\xrightarrow{k^-_{21}}S_1,\;\;\;
	3S_2\xrightarrow{k^-_{22}}S_1+2S_2,
	\end{equation*}
	and the backward reactions are given by
	\begin{equation*}
	S_1\xrightarrow{k^-_{11}}\varnothing,\;\;\;
	2S_1\xrightarrow{k^-_{12}}S_1,\;\;\;
	S_1\xrightarrow{k^+_{21}}S_2,\;\;\;
	S_1+2S_2\xrightarrow{k^+_{22}}3S_2.
	\end{equation*}
	The first two forward reactions have the same reaction vector $\omega_1=(1,0)$  and the last two forward reactions have the same reaction vector $\omega_2=(1,-1)$. The multiplicities of the two reaction vectors  are given by $r_1 = r_2=2$, respectively, and the two reaction vectors are linearly independent. Moreover, we have
	\begin{gather*}
	\nu_{11} = (0,0),\;\;\;\nu'_{11} = (1,0),\;\;\;\nu_{12} = (1,0),\;\;\;\nu'_{12} = (2,0),\\
	\nu_{21} = (0,1),\;\;\;\nu'_{21} = (1,0),\;\;\;\nu_{22} = (0,3),\;\;\;\nu'_{22} = (1,2).
	\end{gather*}
	It is easy to check that $(\nu_{22}^j-\nu^j_{21})\omega_1^j=0$ for $j=1,2$. This shows that the orthogonality condition is satisfied for the reaction vector $\omega_2$. By Theorem \ref{result4}, the reaction network satisfies stochastic detailed balance if the condition (b) holds for $p = 1$, namely $k^+_{11}/k^-_{11}=k^+_{12}/k^-_{12}$. Thus, if $k^+_{11}/k^-_{11}=k^+_{12}/k^-_{12}$ but $k^+_{21}/k^-_{21}\neq k^+_{22}/k^-_{22}$, then the system satisfies stochastic detailed balance but violates deterministic detailed balance.
\end{example}


\section{Global potential for chemical reaction networks}
In this section, we investigate the quasi-potential of chemical reaction networks, which plays a central role in the Freidlin-Wentzell-type metastability theory. The quasi-potential defined in \eqref{quasipotential} has two important features: (i) it is locally defined within each basin of attraction and in general cannot be globally defined over the whole space and (ii) it is defined in the variational form which is usually too complicated to be computed explicitly. The following theorem shows that, under the condition of local detailed balance, the quasi-potential not only can be defined globally over the whole space but also has an explicit and concise expression.

In the following, we always assume that the stochastic model $X^V$ starts from some $x^V_0\in E_V$ which satisfies $x^V_0\rightarrow x_0\in\mathbb{R}^d_{\geq0}$ as $V\rightarrow\infty$. Recall that the critical points of a function $U\in C^1(\mathbb{R}^d_{>0})$ are those points in $\mathbb{R}^d_{>0}$ at which $\nabla U$ vanish.

\begin{theorem}\label{result2}
	Suppose that a reaction network satisfies local detailed balance. Let $\{\omega_{i_1},\cdots,\omega_{i_m}\}$ be an arbitrary basis of ${\rm span}(V(\mathcal{R}))$ and let $M = (\omega_{i_1}^T,\cdots,\omega^T_{i_m})$ be a $d\times m$ matrix, where $m$ is the dimension of ${\rm span}(V(\mathcal{R}))$. Let $F:\mathbb{R}^d_{>0}\rightarrow\mathbb{R}^d$ be a vector field defined as
	\begin{equation}\label{vectorfield}
	F(x) = \left(\log\frac{f^+_{i_1}(x)}{f^-_{i_1}(x)},\cdots,\log \frac{f^+_{i_m}(x)}{f^-_{i_m}(x)}\right)(M^TM)^{-1}M^T.
	\end{equation}
	Then the following five statements hold:\\
	(a) The definition of the vector field $F$ is independent of the choice of the basis of ${\rm span}(V(\mathcal{R}))$. In addition, for any $1\leq p\leq r$, we have
	\begin{equation}\label{crucial}
	F(x)\cdot\omega_p =\log\frac{f^+_p(x)}{f^-_p(x)},\;\;\;x\in\mathbb{R}^d_{>0}.
	\end{equation}
	(b) The vector field $F$ has a potential function $U\in C^\infty(\mathbb{R}^d_{>0})$, namely
	\begin{equation}\label{gradient}
	F(x) = -\nabla U(x),\;\;\;x\in\mathbb{R}^d_{>0}\cap(x_0+\mathrm{span}(V(\mathcal{R}))).
	\end{equation}
	(c) The potential function $U$ satisfies
	\begin{equation*}
	\frac{{\rm d}}{{\rm d}t}U(x(t))\leq 0,\;\;\;t\geq 0,
	\end{equation*}
	where $x = x(t)$ is the solution of the deterministic model \eqref{deterministic}, and the equality holds if and only if the
	deterministic model starts from any one of its equilibrium points.\\
	(d) The critical points of $U$ within $\mathbb{R}^d_{>0}\cap(x_0+\mathrm{span}(V(\mathcal{R})))$ are also the equilibrium points of the deterministic model \eqref{deterministic}.\\
	(e) Let $c\in\mathbb{R}^d_{>0}\cap(x_0+\mathrm{span}(V(\mathcal{R})))$ be an equilibrium point of the deterministic model \eqref{deterministic}. If $y\in\mathbb{R}^d_{>0}$ is attracted to $c$ for the deterministic model \eqref{deterministic}, then
	\begin{equation}\label{potentialdifference}
	W(c,y) = U(y)-U(c),
	\end{equation}
	where $W$ is the quasi-potential defined in \eqref{quasipotential}.
\end{theorem}

\begin{proof}
	The proof of this theorem will be given in Section \ref{result2proof}.
\end{proof}

It is easy to see that if two potential functions $U_1,U_2\in C^\infty(\mathbb{R}^d_{>0})$ both satisfy \eqref{gradient}, namely
\begin{equation*}
F(x) = -\nabla U_1(x) = -\nabla U_2(x),\;\;\;x\in\mathbb{R}^d_{>0}\cap(x_0+\mathrm{span}(V(\mathcal{R}))),
\end{equation*}
then $U_1$ and $U_2$ must coincide with each other (up to a constant) on $\mathbb{R}^d_{>0}\cap(x_0+\mathrm{span}(V(\mathcal{R})))$.

\begin{definition}
	The potential function $U$ introduced in Theorem \ref{result2} is called the \emph{global potential} of a reaction network.
\end{definition}

\begin{remark}
	In the classic Freidlin-Wentzell theory for randomly perturbed dynamical systems \cite{freidlin1998random}, it was shown that when detailed balance is satisfied, the system has a global potential that can be computed explicitly \cite[Chapter 4, Theorem 3.1]{freidlin1998random}. The above theorem is actually the counterpart of this result for stochastic reaction networks. It shows that under the condition of local detailed balance, we are able to construct a global potential of the reaction network, which is exactly the same as the quasi-potential (up to a constant) within each basin of attraction.
\end{remark}

\begin{remark}
	Combining Theorems \ref{result1} and \ref{result2}, we can see that stochastic detailed balance ensures the existence of a global potential. From the deterministic perspective, the vector field $F$ defined in \eqref{vectorfield} can be understood as the chemical force of the deterministic model. Recall that $F$ has a potential function if and only if the line integral of $F$ (the work exerted by the chemical force) along any smooth closed curve is zero. Similarly, from the stochastic perspective, the chemical force of the stochastic model between any two states $x,y\in E_V$ is defined as $\log q^V_{x,y}/q^V_{y,x}$ \cite[Theorem 2.5]{Zhang2012Stochastic}. Then the Kolmogorov cycle condition guarantees that the work exerted by the chemical force along any cycle is zero, namely
	\begin{equation*}
	\log\frac{q^V_{x_1,x_2}}{q^V_{x_2 ,x_1}}+\log\frac{q^V_{x_2,x_3}}{q^V_{x_3,x_2}}+\cdots
	+\log\frac{q^V_{x_n,x_1}}{q^V_{x_1,x_n}} = 0
	\end{equation*}
	for any cycle $x_1\rightarrow x_2\rightarrow\cdots\rightarrow x_n\rightarrow x_1$. This intuitively explains why stochastic detailed balance implies the existence of a global potential.
\end{remark}

The following corollary shows that the global potential $U$ satisfies the \emph{time-independent Hamilton-Jacobi equation}.

\begin{corollary}
	Suppose that a reaction network satisfies local detailed balance. Then we have
	\begin{equation*}
	H(x,-F(x)) = 0,\;\;\;x\in\mathbb{R}^d_{>0},
	\end{equation*}
	where $H(x,\theta)$ is the Hamiltonian defined in \eqref{hamiltonian} and $F(x)$ is the vector field defined in $\eqref{vectorfield}$. In particular, we have
	\begin{equation*}
	H(x,\nabla U(x)) = 0,\;\;\;x\in\mathbb{R}^d_{>0}\cap(x_0+\mathrm{span}(V(\mathcal{R}))),
	\end{equation*}
	where $U(x)$ is the global potential introduced in Theorem \ref{result2}.
\end{corollary}

\begin{proof}
	Since the network satisfies local detailed balance, it follows from Theorem \ref{result2}(a) that
	\begin{equation*}
	\begin{aligned}
	H(x,-F(x))&=\sum_{p=1}^{r}\left[f^+_p (x)\left(e^{-\omega_p\cdot F(x)}-1\right)
	+f^-_p (x)\left(e^{\omega_p\cdot F(x)}-1\right)\right]\\
	&=\sum_{p=1}^{r}\Big[f^+_p(x)\Big(e^{-\log \frac{f^+_p(x)}{f^-_p(x)}}-1\Big)+f^-_p(x)\Big(e^{\log \frac{f^+_p(x)}{f^-_p(x)}}-1\Big)\Big]\\
	&=\sum_{p=1}^{r}\left[f^-_p(x)-f^+_p(x)+f^+_p(x)-f^-_p(x)\right] = 0.
	\end{aligned}
	\end{equation*}
	The rest of the proof follows immediately from Theorem \ref{result2}(b).
\end{proof}

If a reaction network satisfies deterministic detailed balance, then any detailed balanced equilibrium point $c\in\mathbb{R}^d_{>0}$ of the deterministic model \eqref{deterministic} must also be complex balanced (see \cite{anderson2010product} for the detailed definition of this concept). Every complex balanced equilibrium point $c=(c_1,\dots,c_d)\in \mathbb{R}^d_{>0}$ can be used to construct a similar potential
\begin{equation}\label{complex}
\tilde{U}(x)=\sum_{j=1}^d \left[x_j\left(\log x_j-\log c_j-1\right)+c_j\right],\;\;\;
x = (x_1,\dots,x_d)\in\mathbb{R}^d_{>0},
\end{equation}
which turns out to be a Lyapunov function of the deterministic model \cite{horn1972general, anderson2015lyapunov}. The following theorem reveals the relationship between the global potential $U$ introduced in Theorem \ref{result2} and the potential $\tilde{U}$ defined in \eqref{complex}.

\begin{theorem}\label{hatu}
	Suppose that a reaction network satisfies deterministic detailed balance with detailed balanced equilibrium point $c\in\mathbb{R}^d_{>0}$. Let $U$ be the global potential of the reaction network and let $\tilde{U}$ be the Lyapunov function defined in \eqref{complex}. Then $U$ coincides with $\tilde{U}$ (up to a constant) on $\mathbb{R}^d_{>0}\cap (x_0+{\rm span}(V(\mathcal{R})))$.
\end{theorem}

\begin{proof}
	Since $c\in\mathbb{R}^d_{>0}$ is a detailed balanced equilibrium point, it follows from \eqref{detailed} that
	\begin{equation*}
	\log\frac{k^+_{pl}}{k^-_{pl}} = \log c\cdot\omega_p,\;\;\;1\leq p\leq r,\;1\leq l\leq r_p,
	\end{equation*}
	where $\log c=(\log c_1,\dots,\log c_d)$. This shows that for any $x\in\mathbb{R}^d_{>0}$,
	\begin{equation*}
	\nabla \tilde{U}(x)\cdot\omega_p = (\log x-\log c)\cdot\omega_p
	= \log x\cdot\omega_p-\log\frac{k^+_{p1}}{k^-_{p1}}.
	\end{equation*}
	Moreover, it follows from Theorem \ref{result2}(a),(b) and \eqref{noequivalent} that for any $x\in\mathbb{R}^d_{>0}\cap (x_0+{\rm span}(V(\mathcal{R})))$,
	\begin{equation*}
	\nabla U(x)\cdot\omega_p = -F(x)\cdot\omega_p = \log\frac{f^-_p(x)}{f^+_p(x)}
	= \log x\cdot\omega_p-\log\frac{k^+_{p1}}{k^-_{p1}}.
	\end{equation*}
	Combining the above two equations yields
	\begin{equation*}
	\nabla \tilde{U}(x)\cdot\omega_p = \nabla U(x)\cdot\omega_p,\;\;\;x\in\mathbb{R}^d_{>0}\cap (x_0+{\rm span}(V(\mathcal{R}))).
	\end{equation*}
	Then the function $H=\tilde{U}-U$ must satisfy
	\begin{equation*}
	\nabla H(x)\cdot z = 0,\;\;\;x\in\mathbb{R}^d_{>0}\cap (x_0+{\rm span}(V(\mathcal{R}))),\;
	z\in{\rm span}(V(\mathcal{R})).
	\end{equation*}
	For any $x,y\in\mathbb{R}^d_{>0}\cap (x_0+{\rm span}(V(\mathcal{R})))$, let $\phi:[0,1]\rightarrow \mathbb{R}^d_{>0}$ be an arbitrary smooth curve satisfying
	\begin{equation*}
	\phi(0)=x,\;\;\;\phi(1)=y,\;\;\;\phi(\cdot)\in x_0+\mathrm{span}(V(\mathcal{R})).
	\end{equation*}
	Then we have
	\begin{equation*}
	H(y)-H(x)=\int_0^1\nabla H(\phi(t))\cdot\dot{\phi}(t){\rm d}t = 0,
	\end{equation*}
	where we have used the fact that $\dot{\phi}(\cdot)\in {\rm span}(V(\mathcal{R}))$. This implies the desired result.
\end{proof}

From the above theorem, $U$ must coincide with $\tilde{U}$ on $\mathbb{R}^d_{>0}\cap (x_0+{\rm span}(V(\mathcal{R})))$ if the reaction network satisfies deterministic detailed balance. Moreover, Theorem \ref{result2}(b) shows that $U$ is the potential function of the vector field $F$. However, the following counterexample shows that $\tilde{U}$ is not necessarily the potential function of the vector field $F$, even the reaction network satisfies deterministic detailed balance.

\begin{example}
	Consider the following chemical reaction system:
	\begin{equation*}
	S_2{\autorightleftharpoons{$k^+$}{$k^-$}} S_1.
	\end{equation*}
	Clearly, the system satisfies deterministic detailed balance and $V(\mathcal{R})) = \{(1,-1)\}$. It is easy to check that the vector field $F$ is given by
	\begin{equation}\label{fexample}
	F(x_1,x_2)=\frac{1}{2}\log\frac{k^+x_2}{k^-x_1}(1,-1).
	\end{equation}
	Suppose that $x_0 = (1,0)$. Then there is a unique detailed balanced equilibrium point
	\begin{equation*}
	c=\left(\frac{k^+}{k^++k^-},\frac{k^-}{k^++k^-}\right)\in\mathbb{R}^2_{>0}\cap (x_0+{\rm span}(V(\mathcal{R}))).
	\end{equation*}
	Then the Lyapunov function $\tilde{U}$ associated with the equilibrium point $c$ is given by
	\begin{equation*}
	\tilde{U}(x_1,x_2) = x_1\left(\log x_1-\log\frac{k^+}{k^++k^-}-1\right)
	+x_2\left(\log x_2-\log\frac{k^-}{k^++k^-}-1\right)+1.
	\end{equation*}
	This shows that
	\begin{equation*}
	\nabla\tilde{U}(x_1,x_2) = \left(\log x_1-\log\frac{k^+}{k^++k^-},\log x_2-\log\frac{k^-}{k^++k^-}\right).
	\end{equation*}
	It is easy to check that $-\nabla\tilde{U}\neq F$ on $\mathbb{R}^2_{>0}\cap (x_0+{\rm span}(V(\mathcal{R})))$ unless $(x_1,x_2) = c$.
\end{example}

We next give an example showing the application of the results in this section.

\begin{example}
	We revisit the chemical reaction system given in Example \ref{bistable}. Recall that the system satisfies stochastic detailed balance when $k^+_{11}/k^-_{11}=k^+_{12}/k^-_{12}$. Here we assume that this condition is satisfied. Under this condition, it follows from Theorem \ref{result1} that the system also satisfies local detailed balance. Note that the reaction vectors in $V(\mathcal{R})$ are given by $\omega_1 = (1,0)$ and $\omega_2 = (1,-1)$, which are linearly independent. Therefore, the matrix $M$ is given by
	\begin{equation*}
	M = \begin{pmatrix} 1 & 1\\ 0 & -1 \end{pmatrix}=(M^TM)^{-1}M^T,
	\end{equation*}
	and thus the vector field $F$ is given by
	\begin{equation*}
	\begin{aligned}
	F(x) &= \left(\log\frac{f^+_1(x)}{f^-_1(x)},\log\frac{f^+_2(x)}{f^-_2(x)}\right)(M^TM)^{-1}M^T\\
	&=\left(\log\frac{k^+_{11}+k^+_{12}x_1}{k^-_{11}x_1+k^-_{12}x_1^2},\log\frac{k^-_{21}x_2+k^-_{22}x_2^3}{k^+_{21}x_1+k^+_{22}x_1x^2_2}\right)
	\begin{pmatrix} 1 & 1\\ 0 & -1 \end{pmatrix}\\
	&=\left(\log\frac{k^+_{11}}{k^-_{11}x_1},
	\log\frac{k^+_{11}(k^+_{21}+k^+_{22}x^2_2)}{k^-_{11}(k^-_{21}x_2+k^-_{22}x_2^3)}\right),
	\end{aligned}
	\end{equation*}
	where we have used the condition $k^+_{11}/k^-_{11}=k^+_{12}/k^-_{12}$. It then follows from Theorem \ref{result2}(b) that the system has a global potential which satisfies
	\begin{equation*}
	\begin{aligned}
	-\left(\frac{\partial U}{\partial x_1},\frac{\partial U}{\partial x_2}\right)
	= \left(\log\frac{k^+_{11}}{k^-_{11}x_1},
	\log\frac{k^+_{11}(k^+_{21}+k^+_{22}x^2_2)}{k^-_{11}(k^-_{21}x_2+k^-_{22}x_2^3)}\right).
	\end{aligned}
	\end{equation*}
	Integrating the above equation gives the following explicit expression of the global potential:
	\begin{equation*}
	\begin{aligned}
	U(x_1,x_2) =&\; x_1\left(\log\frac{k^-_{11}x_1}{k^+_{11}}-1\right)+x_2\left(\log \frac{k^-_{11}x_2}{k^+_{11}}-1\right)+x_2\log\frac{k^-_{22}x_2^2+k^-_{21}}{k^+_{22}x_2^2+k^+_{21}}\\
	&\;+2\sqrt{\frac{k^-_{21}}{k^-_{22}}}\arctan\left(\sqrt{\frac{k^-_{22}}{k^-_{21}}}x_2\right)-2\sqrt{\frac{k^+_{21}}{k^+_{22}}}\arctan\left(\sqrt{\frac{k^+_{22}}{k^+_{21}}}x_2\right)
	+C,
	\end{aligned}
	\end{equation*}
	where $C$ is a constant which can be chosen so that the minimum of $U$ is zero. By Theorem \ref{result2}(d), any critical point $c = (c_1,c_2)\in\mathbb{R}^2_{>0}$ of the global potential $U$ must satisfy $\nabla U(c)=0$, namely
	\begin{equation*}
	c_1 = \frac{k^+_{11}}{k^-_{11}},\;\;\; k^-_{11}k^-_{22}c_2^3-k^+_{11}k^+_{22}c_2^2+k^-_{11}k^-_{21}c_2-k^+_{11}k^+_{21}=0.
	\end{equation*}
	Note that $c_2$ satisfies a cubic equation, which is capable of having three distinct positive real roots. In this case, the global potential $U$ has two local minimum points $c_A$ and $c_C$ and one saddle point $c_B$, as illustrated in Fig. \ref{bimodality}. It is easy to see that $c_A$ and $c_C$ are stable equilibrium points of the deterministic model \eqref{deterministic} and $c_B$ is an unstable equilibrium point. Therefore, by applying Theorem \ref{result1}, we have constructed a high-dimensional chemical reaction network that both satisfies stochastic detailed balance and displays multistability, i.e. multiple attractors for the deterministic model.
	\begin{figure}
		\centering\includegraphics[width=150mm]{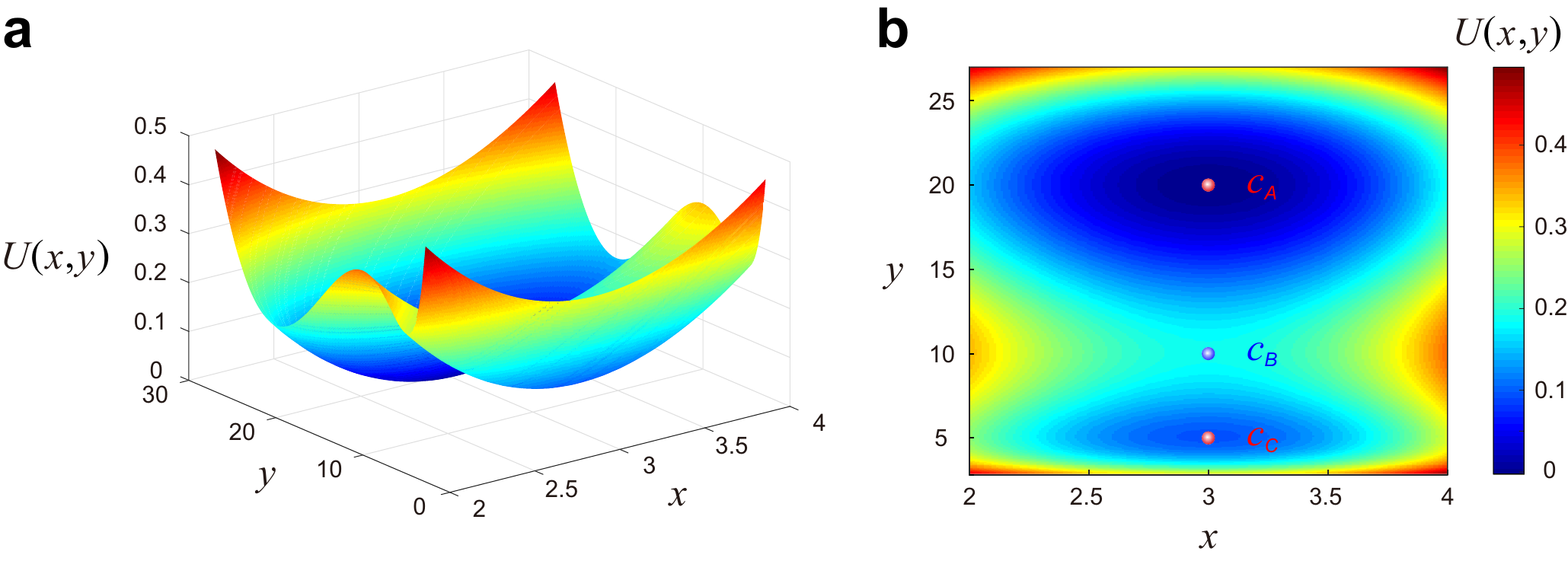}
		\caption{Potential function $U$ of the reaction network given in Example \ref{bistable}. (a) Three-dimensional plot of the potential function $U$. (b) Heat plot of the potential function $U$. The rate constants of the reaction network are chosen as $k^+_{11}=3, k_{11}^-=1, k^+_{12}= 3, k^-_{12}= 1, k^+_{21}=1000/3, k^-_{21}=350, k^+_{22}=35/3, k^-_{22}=1$. The local minimum points of $U$ are given by $c_A=(3,20)$ and $c_C=(3,5)$ and the saddle point of $U$ is given by $c_B=(3,10)$.}\label{bimodality}
	\end{figure}
	
	When $V$ is large, a stochastic reaction network can transition between multiple attractors with low probability events. In analogy to the classic Freidlin-Wentzell theory \cite{freidlin1998random}, if $x_0$ is in the basin of attraction of $c_A$, then for any $\delta>0$, it follows from \cite[Theorem 1.15]{agazzi2018large} that the transition time between attractors has the following asymptotic behavior:
	\begin{equation*}
	\lim_{V\rightarrow\infty}\frac{1}{V}\log\mathbb{E}_{x_0^V}\sigma(B(c_C,\delta)) = W(c_A,c_B),
	\end{equation*}
	where $\sigma(B(c_C,\delta))$ denotes the hitting time of the ball centered at $c_C$ with radius $\delta$. By Theorem \ref{result2}(e), for any $y\in\mathbb{R}^2_{>0}$ staying in the basin of attraction of $c_A$, we have
	\begin{equation*}
	W(c_A,y) = U(y)-U(c_A).
	\end{equation*}
	Taking $y\rightarrow c_B$ in the above equation and applying the continuity of the quasi-potential finally yield
	\begin{equation*}
	\lim_{V\rightarrow\infty}\frac{1}{V}\log\mathbb{E}_{x_0^V}\sigma(B(c_C,\delta)) = U(c_B)-U(c_A).
	\end{equation*}
	Note that the right-hand side of this equation is independent of $c_C$. This is because the trajectory of the stochastic model will be attracted by $c_C$ with very fast speed once it has escaped from the basin of attraction of $c_A$. Therefore, the hitting time of $B(c_C,\delta)$ is mainly determined by the exit time from the basin of attraction of $c_A$.
\end{example}

\begin{remark}
	The reaction networks in Examples \ref{schlogl} and \ref{bistable} are called multistable systems since they are capable of producing multiple positive equilibrium points of the deterministic model. So far, many results have been obtained to identify whether a deterministic reaction network admits multiple equilibrium points \cite{joshi2013atoms, conradi2017identifying, conradi2019multistationarity}. Here we mainly focus on the stochastic model and our results can be applied to investigate the stochastic transitions between multiple attractors.
\end{remark}

\section{Proof of Theorem \ref{result1}}\label{result1proof}
\begin{proof}[Proof of Theorem \ref{result1}]
	By the definition of the transition rates of the stochastic model, it is easy to see that for any $1\leq p\leq r$, $x^V\in E_V$, and $x^V\rightarrow x\in\mathbb{R}^d_{>0}$, we have
	\begin{equation*}
	\begin{split}
	\lim_{V\rightarrow\infty}\frac{q^V_{x^V,x^V+\frac{\omega_p}{V}}}{V}
	&= \lim_{V\rightarrow\infty}\sum_{l=1}^{r_p}k^+_{pl}
	\prod_{j=1}^d\frac{Vx^V_j}{V}\frac{Vx^V_j-1}{V}\dots\frac{Vx^V_j-\nu^j_{pl}+1}{V}\\
	&= \sum_{l=1}^{r_p}k^+_{pl}\prod_{j=1}^dx_j^{\nu^j_{pl}} = f_p^+(x),\\
	\lim_{V\rightarrow\infty}\frac{q^V_{x^V+\frac{\omega_p}{V},x^V}}{V}
	&= \lim_{V\rightarrow\infty}\sum_{l=1}^{r_p}k^-_{pl}
	\prod_{j=1}^d\frac{Vx^V_j+\omega^j_p}{V}\frac{Vx^V_j+\omega^j_p-1}{V}\dots\frac{Vx^V_j+\omega^j_p-\nu'^j_{pl}+1}{V}\\
	&= \sum_{l=1}^{r_p}k^-_{pl}\prod_{j=1}^dx_j^{\nu'^j_{pl}} = f_p^-(x).
	\end{split}
	\end{equation*}
	This clearly shows that
	\begin{equation}\label{eq80}
	\lim_{V\rightarrow \infty}\frac{q^V_{x^V,x^V+\frac{\omega_p}{V}}}{q^V_{x^V+\frac{\omega_p}{V},x^V}}
	= \frac{f_p^+(x)}{f_p^-(x)}.
	\end{equation}
	To prove that the system satisfies local detailed balance, we first prove that zero-order local detailed balance is satisfied. For any sufficiently large $V$ and any integers $\xi_1,\dots,\xi_r$ satisfying $\sum_{p=1}^r \xi_p\omega_p = 0$, we construct the following cycle in $E_V$:
	\begin{equation*}
	\begin{split}
	&\;x^V \rightarrow x^V+\frac{\mathrm{sgn}(\xi_1)\omega_1}{V} \rightarrow\cdots\rightarrow
	x^V+\frac{\xi_1\omega_1}{V} \\
	\rightarrow&\; x^V+\frac{\xi_1\omega_1+\mathrm{sgn}(\xi_2)\omega_2}{V} \rightarrow \cdots \rightarrow
	x^V+\frac{\xi_1\omega_1+\xi_2\omega_2}{V} \rightarrow \cdots \\
	\rightarrow&\; x^V+\frac{\xi_1\omega_1+\xi_2\omega_2+\cdots+\mathrm{sgn}(\xi_r)\omega_r}{V} \rightarrow \cdots \rightarrow x^V+\frac{\xi_1\omega_1+\xi_2\omega_2+\cdots+\xi_r\omega_r}{V} = x^V,
	\end{split}
	\end{equation*}
	where $\mathrm{sgn}(x)$ is the sign function which takes the value of $1$ if $x> 0$, takes the value of $0$ if $x=0$, and takes the value of $-1$ if $x<0$. Applying the Kolmogorov cycle condition to this cycle yields
	\begin{equation*}
	\begin{split}
	&\prod_{l=0}^{|\xi_1|-1}
	\frac{q^V_{x^V+\frac{\mathrm{sgn}(\xi_1)l\omega_1}{V},x^V+\frac{\mathrm{sgn}(\xi_1)(l+1)\omega_1}{V}}}
	{q^V_{x^V+\frac{\mathrm{sgn}(\xi_1)(l+1)\omega_1}{V},x^V+\frac{\mathrm{sgn}(\xi_1)l\omega_1}{V}}}\cdots\\
	&\prod_{l=0}^{|\xi_r|-1}
	\frac{q^V_{x^V+\frac{\xi_1\omega_1+\xi_2\omega_2+\cdots+\mathrm{sgn}(\xi_r)l\omega_r}{V},
			x^V+\frac{\xi_1\omega_1+\xi_2\omega_2+\cdots+\mathrm{sgn}(\xi_r)(l+1)\omega_r}{V}}}
	{q^V_{x^V+\frac{\xi_1\omega_1+\xi_2\omega_2+\cdots+\mathrm{sgn}(\xi_r)(l+1)\omega_r}{V},
			x^V+\frac{\xi_1\omega_1+\xi_2\omega_2+\cdots+\mathrm{sgn}(\xi_r)l\omega_r}{V}}}
	= 1.
	\end{split}
	\end{equation*}
	Taking the limit of $V\rightarrow\infty$ in this equation and applying \eqref{eq80} yield
	\begin{equation*}
	\left(\frac{f_1^+(x)}{f_1^-(x)}\right)^{\xi_1}\cdots\left(\frac{f_r^+(x)}{f_r^-(x)}\right)^{\xi_r} = 1.
	\end{equation*}
	Taking logarithms on both sides yields
	\begin{equation*}
	\sum_{p=1}^r\xi_p\log\frac{f_p^+(x)}{f_p^-(x)} = 0,
	\end{equation*}
	which shows that the system satisfies zero-order local detailed balance.
	
	We next prove that first-order local detailed balance is satisfied. To this end, for any sufficiently large $V$ and any $1\leq p,q\leq r$ with $p\neq q$, we consider the following parallelogram cycle in $E_V$:
	\begin{equation*}
	x^V \rightarrow x^V+\frac{\omega_p}{V} \rightarrow x^V+\frac{\omega_p}{V}+\frac{\omega_{q}}{V} \rightarrow x^V+\frac{\omega_{q}}{V} \rightarrow x^V.
	\end{equation*}
	Applying the Kolmogorov cycle condition to this cycle yields
	\begin{equation*}
	\frac{q^V_{x^V,x^V+\frac{\omega_p}{V}}
		q^V_{x^V+\frac{\omega_p}{V},x^V+\frac{\omega_p}{V}+\frac{\omega_{q}}{V}}
		q^V_{x^V+\frac{\omega_p}{V}+\frac{\omega_{q}}{V},x^V+\frac{\omega_{q}}{V}}
		q^V_{x^V+\frac{\omega_{q}}{V},x^V}}
	{q^V_{x^V+\frac{\omega_p}{V},x^V}
		q^V_{x^V+\frac{\omega_p}{V}+\frac{\omega_{q}}{V},x^V+\frac{\omega_p}{V}}
		q^V_{x^V+\frac{\omega_{q}}{V},x^V+\frac{\omega_p}{V}+\frac{\omega_{q}}{V}}
		q^V_{x^V,x^V+\frac{\omega_{q}}{V}}}
	= 1.
	\end{equation*}
	Taking logarithms on both sides of this equation yields
	\begin{equation*}
	\begin{aligned}
	\textrm{I}_V &:=
	\log\frac{\sum_{l=1}^{r_p}\frac{k^+_{pl}}{V^{|\nu_{pl}|}}\frac{(Vx^V+\omega_{q})!}{(Vx^V+\omega_{q}-\nu_{pl})!}}
	{\sum_{l=1}^{r_p}\frac{k^-_{pl}}{V^{|\nu'_{pl}|}}\frac{(Vx^V+\omega_{q}+\omega_p)!}{(Vx^V+\omega_{q}+\omega_p-\nu'_{pl})!}}
	-\log\frac{\sum_{l=1}^{r_p}\frac{k^+_{pl}}{V^{| \nu_{pl}|}}\frac{(Vx^V)!}{(Vx^V-\nu_{pl})!}}
	{\sum_{l=1}^{r_p}\frac{k^-_{pl}}{V^{|\nu'_{pl}|}}\frac{(Vx^V+\omega_p)!}{(Vx^V+\omega_p-\nu'_{pl})!}}\\
	&=\log\frac{\sum_{l=1}^{r_{q}}\frac{k^+_{ql}}{V^{|\nu_{ql}|}}\frac{(Vx^V+\omega_p)!}{(Vx^V+\omega_p-\nu_{ql})!}}
	{\sum_{l=1}^{r_{q}}\frac{k^-_{ql}}{V^{|\nu'_{ql}|}}\frac{(Vx^V+\omega_p+\omega_{q})!}{(Vx^V+\omega_p+\omega_{q}-\nu'_{ql})!}}
	-\log\frac{\sum_{l=1}^{r_{q}}\frac{k^+_{ql}}{V^{|\nu_{ql}|}}\frac{(Vx^V)!}{(Vx^V-\nu_{ql})!}}
	{\sum_{l=1}^{r_{q}}\frac{k^-_{ql}}{V^{|\nu'_{ql}|}}\frac{(Vx^V+\omega_{q})!}{(Vx^V+\omega_{q}-\nu'_{ql})!}}
	:= \textrm{II}_V.
	\end{aligned}
	\end{equation*}
	By the mean value theorem, it is not hard to prove that
	\begin{gather*}
	\lim_{V\rightarrow \infty}\frac{\textrm{I}_V}{\frac{1}{V}}
	= \omega_{q}\cdot\nabla\log\frac{f^+_p(x)}{f^-_p(x)},\\
	\lim_{V\rightarrow \infty}\frac{\textrm{II}_V}{\frac{1}{V}}
	= \omega_p\cdot\nabla\log\frac{f^+_{q}(x)}{f^-_{q}(x)}.
	\end{gather*}
	Combining the above two equations yields
	\begin{equation*}
	\omega_{q}\cdot\nabla\log\frac{f^+_p(x)}{f^-_p(x)}
	= \omega_p\cdot\nabla\log\frac{f^+_{q}(x)}{f^-_{q}(x)},
	\end{equation*}
	which shows that first-order local detailed balance is also satisfied.
\end{proof}

\section{Proof of Theorem \ref{result4}}\label{result4proof}
To prove that a reaction network satisfies stochastic detailed balance, it suffices to prove that for any $V>0$ and any cycle $\eta_1\rightarrow\eta_2\rightarrow\cdots\rightarrow \eta_L\rightarrow\eta_1$ in $E_V$, the Kolmogorov cycle condition
\begin{equation*}
q^V_{\eta_1,\eta_2}q^V_{\eta_2,\eta_3}\cdots q^V_{\eta_L,\eta_1}
= q^V_{\eta_2,\eta_1}q^V_{\eta_3,\eta_2}\cdots q^V_{\eta_1,\eta_L}
\end{equation*}
is satisfied. To this end, we need the following lemma.

\begin{lemma}\label{tool2}
	Under the conditions in Theorem \ref{result4}, for any $x\in E_V$ and any integers $\xi_1,\cdots,\xi_r$, whenever $q^V_{x,x+\frac{\omega_p}{V}}>0$ and $q^V_{\tilde{x},\tilde{x}+\frac{\omega_p}{V}}>0$ for some $1\leq p\leq r$, we have
	\begin{equation}\label{step2}
	\frac{q^V_{x,x+\frac{\omega_p}{V}}}{q^V_{x+\frac{\omega_p}{V},x}}\cdot\frac{(Vx+\omega_p)!}{(Vx)!}
	= \frac{q^V_{\tilde{x},\tilde{x}+\frac{\omega_p}{V}}}
	{q^V_{\tilde{x}+\frac{\omega_p}{V},\tilde{x}}}\cdot\frac{(V\tilde{x}+\omega_p)!}{(V\tilde{x})!},
	\end{equation}
	where
	\begin{equation*}
	\tilde{x} = x+\sum_{q\neq p}\xi_q\frac{\omega_{q}}{V}.
	\end{equation*}
\end{lemma}

\begin{proof}
	We first discuss the relationship between the reactions that can occur at $x$ and the reactions that can occur at $\tilde{x}$. For each $1\leq p\leq r$, let
	\begin{equation*}
	\mathcal{R}^+_p(x)
	= \{R^+_{pl}\in \mathcal{R}^+_p:Vx_j\geq \nu^j_{pl}\;\text{for all}\;1\leq j\leq d\}
	\end{equation*}
	be the family of reactions that belong to $\mathcal{R}^+_p$ and can occur at $x$. We claim that if $\omega_p$ satisfies the orthogonality condition, $\mathcal{R}^+_p(x)\neq\varnothing$, and $\mathcal{R}^+_p(\tilde{x})\neq\varnothing$, then $\mathcal{R}^+_p(x)=\mathcal{R}^+_p(\tilde{x})$.
	
	To prove this, set
	\begin{equation*}
	J_p=\{1\leq j\leq d:\nu^j_{pl_1} = \nu^j_{pl_2}\;\text{for all}\;1\leq l_1,l_2\leq r_p\}.
	\end{equation*}
	Since $\mathcal{R}^+_p(x)\neq\varnothing$, there exists $1\leq l_1\leq r_p$ such that $Vx_j\geq\nu^j_{pl_1}$ for all $1\leq j\leq d$. It then follows from the definition of $J_p$ that $Vx_j\geq\nu^j_{pl}$ for all $j\in J_p$ and $1\leq l\leq r_p$. Similarly, since $\mathcal{R}^+_p(\tilde{x})\neq\varnothing$, we conclude that $V\tilde{x}_j\geq \nu^j_{pl}$ for all $j\in J_p$ and $1\leq l\leq r_p$. On the other hand, if $\omega_p$ satisfies the orthogonality condition, then $\omega^j_{q} = 0$ for all $j\notin J_p$ and $q\neq p$. This shows that
	\begin{equation*}
	\tilde{x}_j = x_j+\sum_{q\neq p}\xi_q\frac{\omega^j_{q}}{V} = x_j
	\end{equation*}
	for all $j\notin J_p$. Therefore, for all $1\leq j\leq d$ and $1\leq l\leq r_p$, we have proved that $Vx_j\geq \nu^j_{pl}$ holds if and only if $V\tilde{x}_j\geq \nu^j_{pl}$ holds. This clearly shows that $\mathcal{R}^+_p(x) = \mathcal{R}^+_p(\tilde{x})$.
	
	We are now in a position to prove \eqref{step2}. For each $1\leq p\leq r$, note that
	\begin{equation*}
	\begin{aligned}
	\frac{q^V_{x,x+\frac{\omega_p}{V}}}{q^V_{x+\frac{\omega_p}{V},x}}\cdot\frac{(Vx+\omega_p)!}{(Vx)!}
	&=\frac{\sum_{l:R^+_{pl}\in\mathcal{R}^+_p(x)}\frac{k^+_{pl}}{V^{|\nu_{pl}|-1}}\frac{(Vx)!}{(Vx-\nu_{pl})!}}
	{\sum_{l:R^+_{pl}\in\mathcal{R}^+_p(x)}\frac{k^-_{pl}}{V^{|\nu'_{pl}|-1}}\frac{(Vx+\omega_p)!}{(Vx+\omega_p-\nu'_{pl})!}}
	\frac{(Vx+\omega_p)!}{(Vx)!}\\
	&=\frac{\sum_{l:R^+_{pl}\in\mathcal{R}^+_p(x)}\frac{k^+_{pl}}{V^{|\nu_{pl}|-1}}\frac{1}{(Vx-\nu_{pl})!}}
	{\sum_{l:R^+_{pl}\in\mathcal{R}^+_p(x)}\frac{k^-_{pl}}{V^{|\nu'_{pl}|-1}}\frac{1}{(Vx-\nu_{pl})!}},
	\end{aligned}
	\end{equation*}
	where we have used the fact that $\omega_p-\nu'_{pl}=-\nu_{pl}$ for any $1\leq p\leq r$ and $1\leq l\leq r_p$. Since $q^{V}_{x,x+\frac{\omega_p}{V}}>0$ and $q^V_{\tilde{x},\tilde{x}+\frac{\omega_p}{V}}>0$, we have $\mathcal{R}^+_p(x)\neq\varnothing$ and $\mathcal{R}^+_p(\tilde{x})\neq\varnothing$. This shows that $\mathcal{R}^+_p(x) = \mathcal{R}^+_p(\tilde{x})$. Similarly, we have
	\begin{equation*}
	\begin{aligned}
	\frac{q^V_{\tilde{x},\tilde{x}+\frac{\omega_p}{V}}}{q^V_{\tilde{x}+\frac{\omega_p}{V},\tilde{x}}}\cdot\frac{(V\tilde{x}+\omega_p)!}{(V\tilde{x})!}
	&=\frac{\sum_{l:R^+_{pl}\in\mathcal{R}^+_p(\tilde{x})}\frac{k^+_{pl}}{V^{|\nu_{pl}|-1}}\frac{(V\tilde{x})!}{(V\tilde{x}-\nu_{pl})!}}
	{\sum_{l:R^+_{pl}\in\mathcal{R}^+_p(\tilde{x})}\frac{k^-_{pl}}{V^{|\nu'_{pl}|-1}}\frac{(V\tilde{x}+\omega_p)!}{(V\tilde{x}+\omega_p-\nu'_{pl})!}}
	\frac{(V\tilde{x}+\omega_p)!}{(V\tilde{x})!}\\
	&=\frac{\sum_{l:R^+_{pl}\in\mathcal{R}^+_p(x)}\frac{k^+_{pl}}{V^{|\nu_{pl}|-1}}\frac{1}{(V\tilde{x}-\nu_{pl})!}}
	{\sum_{l:R^+_{pl}\in\mathcal{R}^+_p(x)}\frac{k^-_{pl}}{V^{|\nu'_{pl}|-1}}\frac{1}{(V\tilde{x}-\nu_{pl})!}}.
	\end{aligned}
	\end{equation*}
	To prove \eqref{step2}, we only need to prove
	\begin{equation*}
	\frac{\sum_{l:R^+_{pl}\in\mathcal{R}^+_p(x)}\frac{k^+_{pl}}{V^{|\nu_{pl}|-1}}\frac{1}{(Vx-\nu_{pl})!}}
	{\sum_{l:R^+_{pl}\in\mathcal{R}^+_p(x)}\frac{k^-_{pl}}{V^{|\nu'_{pl}|-1}}\frac{1}{(Vx-\nu_{pl})!}}
	=\frac{\sum_{l:R^+_{pl}\in\mathcal{R}^+_p(x)}\frac{k^+_{pl}}{V^{|\nu_{pl}|-1}}\frac{1}{(V\tilde{x}-\nu_{pl})!}}
	{\sum_{l:R^+_{pl}\in\mathcal{R}^+_p(x)}\frac{k^-_{pl}}{V^{|\nu'_{pl}|-1}}\frac{1}{(V\tilde{x}-\nu_{pl})!}}.
	\end{equation*}
	Since  $\nu'_{pl}=\nu_{pl}+\omega_p$ for any $1\leq p\leq r$ and $1\leq l \leq r_p$, we only need to prove
	
	\begin{equation}\label{simplified}
	\frac{\sum_{l:R^+_{pl}\in\mathcal{R}^+_p(x)}\frac{k^+_{pl}}{V^{|\nu_{pl}|}}\frac{1}{(Vx-\nu_{pl})!}}
	{\sum_{l:R^+_{pl}\in\mathcal{R}^+_p(x)}\frac{k^-_{pl}}{V^{|\nu_{pl}|}}\frac{1}{(Vx-\nu_{pl})!}}
	=\frac{\sum_{l:R^+_{pl}\in\mathcal{R}^+_p(x)}\frac{k^+_{pl}}{V^{|\nu_{pl}|}}\frac{1}{(V\tilde{x}-\nu_{pl})!}}
	{\sum_{l:R^+_{pl}\in\mathcal{R}^+_p(x)}\frac{k^-_{pl}}{V^{|\nu_{pl}|}}\frac{1}{(V\tilde{x}-\nu_{pl})!}}.
	\end{equation}
	If $\omega_p$ does not satisfy the orthogonality condition, then \eqref{joshib} must hold. In this case, it is easy to see that
	\begin{equation*}
	\frac{\sum_{l:R^+_{pl}\in\mathcal{R}^+_p(x)}\frac{k^+_{pl}}{V^{|\nu_{pl}|}}\frac{1}{(Vx-\nu_{pl})!}}
	{\sum_{l:R^+_{pl}\in\mathcal{R}^+_p(x)}\frac{k^-_{pl}}{V^{|\nu_{pl}|}}\frac{1}{(Vx-\nu_{pl})!}}
	=\frac{k^+_{p1}}{k^-_{p1}} =\frac{\sum_{l:R^+_{pl}\in\mathcal{R}^+_p(x)}\frac{k^+_{pl}}{V^{|\nu_{pl}|}}\frac{1}{(V\tilde{x}-\nu_{pl})!}}
	{\sum_{l:R^+_{pl}\in\mathcal{R}^+_p(x)}\frac{k^-_{pl}}{V^{|\nu_{pl}|}}\frac{1}{(V\tilde{x}-\nu_{pl})!}}.
	\end{equation*}
	If $\omega_p$ satisfies the orthogonality condition, then we have
	\begin{equation}\label{eq3}
	\begin{aligned}
	&\;\frac{\sum_{l:R^+_{pl}\in\mathcal{R}^+_p(x)}\frac{k^+_{pl}}{V^{|\nu_{pl}|}}\frac{1}{(V\tilde{x}-\nu_{pl})!}}
	{\sum_{l:R^+_{pl}\in\mathcal{R}^+_p(x)}\frac{k^-_{pl}}{V^{|\nu_{pl}|}}\frac{1}{(V\tilde{x}-\nu_{pl})!}}\\ =&\;\frac{\sum_{l:R^+_{pl}\in\mathcal{R}^+_p(x)}\frac{k^+_{pl}}{V^{|\nu_{pl}|}}\frac{1}{\prod_{j\in J_p}(V\tilde{x}_j-\nu^j_{pl})!\prod_{j\notin J_p}(V\tilde{x}_j-\nu^j_{pl})!}}{\sum_{l:R^+_{pl}\in\mathcal{R}^+_p(x)}\frac{k^-_{pl}}{V^{|\nu_{pl}|}}\frac{1}{\prod_{j\in J_p}(V\tilde{x}_j-\nu^j_{pl})!\prod_{j\notin J_p}(V\tilde{x}_j-\nu^j_{pl})!}}\\
	=&\;\frac{\sum_{l:R^+_{pl}\in\mathcal{R}^+_p(x)}\frac{k^+_{pl}}{V^{|\nu_{pl}|}}\frac{1}{\prod_{j\notin J_p}(Vx_j-\nu^j_{pl})!}}{\sum_{l:R^+_{pl}\in\mathcal{R}^+_p(x)}\frac{k^-_{pl}}{V^{|\nu_{pl}|}}\frac{1}{\prod_{j\notin J_p}(Vx_j-\nu^j_{pl})!}
	}=\frac{\sum_{l:R^+_{pl}\in\mathcal{R}^+_p(x)}\frac{k^+_{pl}}{V^{|\nu_{pl}|}}\frac{1}{\prod_{j=1}^d(Vx_j-\nu^j_{pl})!}}{\sum_{l:R^+_{pl}\in\mathcal{R}^+_p(x)}\frac{k^-_{pl}}{V^{|\nu_{pl}|}}\frac{1}{\prod_{j=1}^d (Vx_j-\nu^j_{pl})!}}\\
	=&\;\frac{\sum_{l:R^+_{pl}\in\mathcal{R}^+_p(x)}\frac{k^+_{pl}}{V^{|\nu_{pl}|}}\frac{1}{(Vx-\nu_{pl})!}}{\sum_{l:R^+_{pl}\in\mathcal{R}^+_p(x)}\frac{k^-_{pl}}{V^{|\nu_{pl}|}}\frac{1}{(Vx-\nu_{pl})!}},
	\end{aligned}
	\end{equation}
	where the second and third equalities in \eqref{eq3} follow from the fact that $\nu^j_{pl_1}=\nu^j_{pl_2}$ for any $1\leq l_1,l_2\leq r_p$ and $j\in J_p$ and the fact that $\tilde{x}_j=x_j$ for any $j\notin J_p$. Therefore, we have proved \eqref{simplified}. This completes the proof.
\end{proof}

We are now in a position to prove Theorem \ref{result4}.

\begin{proof}[Proof of Theorem \ref{result4}]
	Let $\eta_1\rightarrow\eta_2\rightarrow\cdots\rightarrow\eta_L\rightarrow\eta_1$ be an arbitrary cycle in $E_V$. We first prove that there is a one-to-one correspondence between the transitions in the cycle. Since $\omega_1,\cdots,\omega_r$ are linearly independent, for each $1\leq p \leq r$, the number of transitions resulting from the reaction vector $\omega_p$ in the cycle must be equal to the number of transitions resulting from the reaction vector $-\omega_p$. Otherwise, the reaction vector $\omega_p$ can be linearly expressed by other reaction vectors in $V(\mathcal{R})$. This contradicts the fact that the elements in $V(\mathcal{R})$ are linearly independent.
	
	We then pair the transitions resulting from the reaction vector $\omega_p$ with the transitions resulting from the reaction vector $-\omega_p$ in the following manner. First, we project the cycle onto the one-dimensional line spanned by $\omega_p$, as illustrated in Fig. \ref{decomposition}. Note that the projection mentioned here means oblique projection rather than orthogonal projection (suppose that $\omega_1,\cdots,\omega_r$ are linearly independent vectors; if a vector $\omega$ can be linearly expressed by $\omega_1,\cdots,\omega_r$ as $\omega = a_1\omega_1+\cdots+a_r\omega_r$, then the (oblique) projection of $\omega$ onto the direction of $\omega_p$ is simply defined as $a_p\omega_p$). The image of the projection onto the one-dimensional line is still a cycle and only the transitions resulting from $\pm\omega_p$ exist in the projected cycle. Second, we decompose the transitions in the projected cycle into multiple floors and we pair each transition with one of its reversed transitions located on the same floor, as depicted in Fig. \ref{decomposition}. This is always possible since the number of transitions resulting from $\omega_p$ on each floor must be equal to that resulting from $-\omega_p$. Based on this method, each transition
	$$
	\eta_i\rightarrow\eta_{i+1}=\eta_i\pm\frac{\omega_p}{V}
	$$
	can be paired with a transition
	$$
	\eta_i+\sum_{q\neq p}\xi_q\frac{\omega_q}{V}\pm\frac{\omega_p}{V}
	= \eta_j\rightarrow\eta_{j+1}
	= \eta_i+\sum_{q\neq p}\xi_q\frac{\omega_q}{V}
	$$
	for some integers $\xi_1,\cdots,\xi_r$. Obviously, this method establishes a one-to-one correspondence between the transitions in the cycle. Applying Lemma \ref{tool2} to the paired transitions yields
	\begin{equation*}
	\frac{q^V_{\eta_{i},\eta_{i+1}}}{q^V_{\eta_{i+1},\eta_{i}}}\frac{(V\eta_{i+1})!}{(V\eta_{i})!}
	=\frac{q^V_{\eta_{j+1},\eta_{j}}}{q^V_{\eta_{j},\eta_{j+1}}}\frac{(V\eta_{j})!}{(V\eta_{j+1})!}.
	\end{equation*}
	This shows that
	\begin{equation*}
	\frac{q^V_{\eta_{i},\eta_{i+1}}q^V_{\eta_{j},\eta_{j+1}}}{q^V_{\eta_{i+1},\eta_{i}}q^V_{\eta_{j+1},\eta_{j}}}
	=\frac{(V\eta_{i})!(V\eta_{j})!}{(V\eta_{i+1})!(V\eta_{j+1})!}.
	\end{equation*}
	Due to the one-to-one correspondence between the transitions in the cycle, we finally obtain
	\begin{equation*}
	\frac{q^V_{\eta_1,\eta_2}q^V_{\eta_2,\eta_3}\cdots q^V_{\eta_L,\eta_1}}
	{q^V_{\eta_2,\eta_1}q^V_{\eta_3,\eta_2}\cdots q^V_{\eta_1,\eta_L}}
	= \frac{(V\eta_1)!(V\eta_2)!\cdots(V\eta_L)!}
	{(V\eta_2)!(V\eta_3)!\cdots(V\eta_1)!} = 1.
	\end{equation*}
	Thus we have proved that the Kolmogorov cycle condition is satisfied for each cycle, which implies that the system satisfies stochastic detailed balance.
	\begin{figure}[htb]
		\centering\includegraphics[width=140mm]{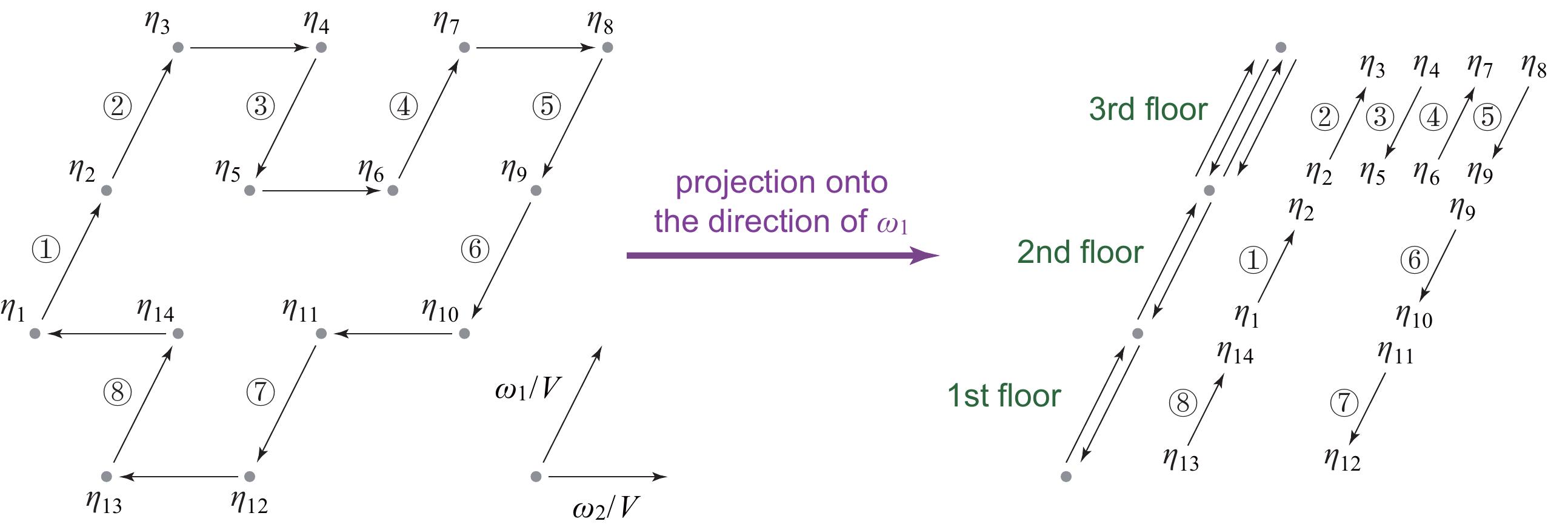}
		\caption{Pairing of the transitions in a cycle resulting from the reaction vector $\omega_1$ with the transitions resulting from the reaction vector $-\omega_1$. The cycle to the left can be projected onto the direction of $\omega_1$, forming the projected cycle (the projection mentioned here means oblique projection rather than orthogonal projection). The projected cycle can be decomposed into many floors. Each transition in the projected cycle can then be paired with one of its reversed transitions located on the same floor. This establishes a one-to-one correspondence between the transitions resulting from $\omega_1$ and the transitions resulting from $-\omega_1$. In this way, $\eta_1\rightarrow \eta_2$ is paired with $\eta_9\rightarrow \eta_{10}$, $\eta_2\rightarrow \eta_3$ is paired with $\eta_4\rightarrow \eta_5$, $\eta_6\rightarrow \eta_7$ is paired with $\eta_8\rightarrow \eta_9$, and $\eta_{11}\rightarrow \eta_{12}$ is paired with $\eta_{13}\rightarrow \eta_{14}$.}\label{decomposition}
	\end{figure}
\end{proof}

\section{Proof of Theorem \ref{result2}}\label{result2proof}
To prove Theorem \ref{result2}, we need some lemmas. The following lemma is exactly Theorem \ref{result2}(a).

\begin{lemma}\label{le10}
	Suppose that a reaction network satisfies zero-order local detailed balance. Then the vector field $F$ defined in \eqref{vectorfield} is independent of the choice of the basis of ${\rm span}(V(\mathcal{R}))$. Moreover, for any $1\leq p\leq r$, we have
	\begin{equation*}
	F(x)\cdot\omega_p =\log\frac{f^+_p(x)}{f^-_p(x)},\;\;\;x\in\mathbb{R}^d_{>0}.
	\end{equation*}
\end{lemma}

\begin{proof}[Proof of Lemma \ref{le10}]
	Let $\{\omega_{i_1},\dots,\omega_{i_m}\}$ and $\{\omega_{j_1},\dots,\omega_{j_m}\}$ be two arbitrary bases of ${\rm span}(V(\mathcal{R}))$ and let $M =(\omega^T_{i_1},\dots,\omega^T_{i_m})$ and $\tilde{M} = (\omega^T_{j_1},\dots,\omega^T_{j_m})$ be two matrices. Since the columns of $M$, as well as the columns of $\tilde{M}$, are linearly independent, there exists an invertible matrix  $A=(a_{pk})\in M_{m\times m}(\mathbb{\mathbb{R}})$ such that $\tilde{M}=MA$, where $M_{m\times m}(\mathbb{\mathbb{R}})$ denotes the set of $m\times m$ matrices whose components are all real numbers. Since $\mathrm{rank}(M) = m$, the matrix $M$ must have an invertible $m\times m$ submatrix. Since the entries of $M$ and $\tilde{M}$ are all rational numbers, using Cramer's rule of computing the inverse matrix, it is easy to see that the entries of $A$ are all rational numbers. Note that each column of $\tilde{M}$ can be linearly expressed by the columns of $M$ as
	\begin{equation*}
	\omega_{j_k}=\sum_{p=1}^m a_{pk}\omega_{i_p},\;\;\;1\leq k\leq m.
	\end{equation*}
	Since the system satisfies zero-order local detailed balance and since $a_{pk}\in\mathbb{Q}$ for all $1\leq p,k\leq m$, we obtain
	\begin{equation*}
	\log\frac{f^+_{j_k}(x)}{f^-_{j_k}(x)}=\sum_{p=1}^m a_{pk}\log \frac{f^+_{i_p}(x)}{f^-_{i_p}(x)},\;\;\;1\leq k\leq m.
	\end{equation*}
	Since $\mathrm{rank}(A) = \mathrm{rank}(A^TA)$ for an arbitrary real matrix $A$ \cite[Section 3.5]{zhang2011matrix}, it is easy to see that the matrices $M^TM$ and $\tilde{M}^T\tilde{M}$ are both invertible. Thus we obtain
	\begin{equation*}
	\begin{aligned}
	&\;\left(\log\frac{f^+_{j_1}(x)}{f^-_{j_1}(x)},\dots,\log\frac{f^+_{j_m}(x)}{f^-_{j_m}(x)}\right)
	(\tilde{M}^T\tilde{M})^{-1}\tilde{M}^T\\
	=&\;\left(\log\frac{f^+_{j_1}(x)}{f^-_{j_1}(x)},\dots,\log\frac{f^+_{j_m}(x)}{f^-_{j_m}(x)}\right)
	A^{-1}(M^TM)^{-1}M^T\\
	=&\;\left(\log\frac{f^+_{i_1}(x)}{f^-_{i_1}(x)},\dots,\log\frac{f^+_{i_m}(x)}{f^-_{i_m}(x)}\right) (M^TM)^{-1}M^T.
	\end{aligned}
	\end{equation*}
	This implies that the definition of the vector field $F$ is independent of the choice of the basis of ${\rm span}(V(\mathcal{R}))$.
	
	On the other hand, since $\{\omega_{i_1},\cdots,\omega_{i_m}\}$ is a basis of ${\rm span}(V(\mathcal{R}))$, for each $1\leq p\leq r$, the reaction vector $\omega_p$ can be linearly expressed by $\omega_{i_1},\cdots,\omega_{i_m}$ as
	\begin{equation*}
	\omega_p = a_1\omega_{i_1}+\cdots+a_m\omega_{i_m} = aM^T,
	\end{equation*}
	where $a = (a_1,\cdots,a_m)\in\mathbb{R}^m$. This shows that
	\begin{equation*}
	(\omega_{i_1}^T,\cdots,\omega_{i_m}^T)(M^TM)^{-1}M^T\omega_p^T = M(M^TM)^{-1}M^TMa^T = Ma^T = \omega_p^T.
	\end{equation*}
	Since the system satisfies zero-order local detailed balance and since the entries of $(M^TM)^{-1}M^T\omega_p^T$ are all rational numbers, we immediately obtain
	\begin{equation}
	\begin{split}
	F(x)\cdot\omega_p &= \left(\log\frac{f^+_{i_1}(x)}{f^-_{i_1}(x)},\cdots,
	\log\frac{f_{i_m}^+(x)}{f^-_{i_m}(x)}\right)(M^TM)^{-1}M^T\omega_p^T\\
	&= \log\frac{f^+_p(x)}{f^-_p(x)},\;\;\;x\in\mathbb{R}^d_{>0}.
	\end{split}
	\end{equation}
	This completes the proof.
\end{proof}

The following lemma is exactly Theorem \ref{result2}(b).

\begin{lemma}\label{le3}
	Suppose that a reaction network satisfies local detailed balance. Then the vector field $F$ has a potential function $U\in C^\infty(\mathbb{R}^d_{>0})$, namely
	\begin{equation*}
	F(x) = -\nabla U(x),\;\;\;x\in\mathbb{R}^d_{>0}\cap(x_0+\mathrm{span}(V(\mathcal{R}))).
	\end{equation*}
\end{lemma}

\begin{proof}
	For any $x\in x_0+{\rm span}(V(\mathcal{R}))$, there exists $y = (y_1,\cdots,y_m)\in\mathbb{R}^m$ such that
	\begin{equation*}
	x = x_0+y_1\omega_{i_1}+\cdots+y_m\omega_{i_m} = x_0+yM^T.
	\end{equation*}
	Since $\{\omega_{i_1},\cdots,\omega_{i_m}\}$ is a basis of ${\rm span}(V(\mathcal{R}))$, this equation establishes a one-to-one correspondence between $x\in\mathbb{R}^d_{>0}\cap(x_0+{\rm span}(V(\mathcal{R})))$ and $y\in E$, where $E$ is some convex open subset of $\mathbb{R}^m$. We denote this correspondence by
	\begin{equation*}
	g:E\rightarrow\mathbb{R}^d_{>0}\cap(x_0+{\rm span}(V(\mathcal{R}))),\;\;\;g(y) = x_0+yM^T.
	\end{equation*}
	The inverse of this affine transformation is then given by
	\begin{equation*}
	g^{-1}:\mathbb{R}^d_{>0}\cap(x_0+{\rm span}(V(\mathcal{R})))\rightarrow E,\;\;\;g^{-1}(x) = (x-x_0)M(M^TM)^{-1}.
	\end{equation*}
	For convenience, we introduce a function $h:E\rightarrow\mathbb{R}^m$ as
	\begin{equation*}
	h(y) = \left(\log\frac{f^+_{i_1}(g(y))}{f^-_{i_1}(g(y))},\cdots,
	\log\frac{f^+_{i_m}(g(y))}{f^-_{i_m}(g(y))}\right).
	\end{equation*}
	It is then easy to check that
	\begin{equation*}
	\partial_lh_k(y)
	= \sum_{p=1}^d\partial_p\left[\log\frac{f^+_{i_k}(g(y))}{f^-_{i_k}(g(y))}\right]\omega_{i_l}^p
	= \omega_{i_l}\cdot\nabla\log\frac{f^+_{i_k}(g(y))}{f^-_{i_k}(g(y))}.
	\end{equation*}
	Since the system satisfies first-order local detailed balance, we immediately obtain
	\begin{equation}\label{symmetric}
	\partial_lh_k(y) = \partial_kh_l(y),\;\;\;1\leq k,l\leq m.
	\end{equation}
	To proceed, we construct a smooth differential 1-form on $E$ as
	\begin{equation*}
	\omega = \sum_{k=1}^mh_k(y){\rm d}y_k.
	\end{equation*}
	It then follows from \eqref{symmetric} that
	\begin{equation*}
	{\rm d}\omega = \sum_{k=1}^m\sum_{l=1}^m\partial_lh_k(y){\rm d}y_l\wedge {\rm d}y_k
	= \sum_{k<l}[\partial_kh_l(y)-\partial_lh_k(y)]{\rm d}y_k\wedge {\rm d}y_l = 0.
	\end{equation*}
	This shows that the differential form $\omega$ is closed. Since $E$ is a convex open subset of $\mathbb{R}^m$, it then follows from the Poincar\'{e} lemma \cite[Theorem 17.14]{lee2013introduction} that the $k$th de Rham cohomology of $E$ is vanishing for each $k\geq 1$, which means that $\omega$ is exact. In other words, there exists a function $\tilde{U}\in C^\infty(E)$ such that
	\begin{equation*}
	\omega = -{\rm d}\tilde{U} = -\sum_{k=1}^m\partial_k\tilde{U}(y){\rm d}y_k.
	\end{equation*}
	This clearly shows that $h = -\nabla\tilde{U}$. Since $\tilde{U}$ is a smooth function on the open set $E\subset\mathbb{R}^m$, we can extend it to a function $\tilde{U}\in C^\infty(\mathbb{R}^m)$. To proceed, we define the potential function $U\in C^\infty(\mathbb{R}^d_{>0})$ as
	\begin{equation*}
	U(x) = \tilde{U}((x-x_0)M(M^TM)^{-1}),\;\;\;x\in\mathbb{R}^d_{>0}.
	\end{equation*}
	For any $x\in\mathbb{R}^d_{>0}\cap(x_0+{\rm span}(V(\mathcal{R})))$, straightforward computations show that
	\begin{equation*}
	\begin{split}
	-\nabla U(x) &= -\nabla\tilde{U}(g^{-1}(x))(M^TM)^{-1}M^T = h(g^{-1}(x))(M^TM)^{-1}M^T\\
	&= \left(\log\frac{f^+_{i_1}(x)}{f^-_{i_1}(x)},\cdots,
	\log\frac{f^+_{i_m}(x)}{f^-_{i_m}(x)}\right)(M^TM)^{-1}M^T = F(x).
	\end{split}
	\end{equation*}
	This completes the proof.
\end{proof}

The following lemma is exactly Theorem \ref{result2}(c),(d).

\begin{lemma}\label{le4}
	Suppose that a reaction network satisfies local detailed balance. Then
	\begin{equation*}
	\frac{{\rm d}}{{\rm d}t}U(x(t))\leq 0,\;\;\;t\geq 0,
	\end{equation*}
	where $x = x(t)$ is the solution of the deterministic model \eqref{deterministic}, and the equality holds if and only if the deterministic model starts from any one of its equilibrium points. Moreover, the critical points of $U$ within $\mathbb{R}^d_{>0}\cap(x_0+\mathrm{span}(V(\mathcal{R})))$ are also the equilibrium points of the deterministic model.
\end{lemma}

\begin{proof}
	Let $x = x(t)$ be the solution of the deterministic model \eqref{deterministic}. It then follows from Lemma \ref{le10} that
	\begin{equation}\label{compare3}
	\begin{aligned}
	\frac{{\rm d}}{{\rm d}t}U(x(t)) &= \nabla U(x(t))\cdot\dot{x}(t)\\
	&= -\sum_{p=1}^r[f^+_p(x(t))-f_p^-(x(t))]F(x(t))\cdot\omega_p\\
	&= -\sum_{p=1}^r[f_p^+(x(t))-f^-_p(x(t))]\log\frac{f_p^+(x(t))}{f_p^-(x(t))}\leq 0,
	\end{aligned}
	\end{equation}
	where the equality holds for some $t\geq 0$ if and only if $f_p^+(x(t))=f^-_p(x(t))$ for all $1\leq p\leq r$, which implies that $x(t)$ is an equilibrium point of the deterministic model \eqref{deterministic}. If the deterministic model does not start from any one of its equilibrium points, then $x(t)$ is not an equilibrium point and thus the equality in \eqref{compare3} cannot be attained. Finally, if $c\in\mathbb{R}^d_{>0}\cap(x_0+\mathrm{span}(V(\mathcal{R})))$ is a critical point of $U$, then we have $F(c) = -\nabla U(c) = 0$. It then follows from \eqref{crucial} that $f^+_p(c) = f^-_p(c)$ for each $1\leq p\leq r$, which shows that $c$ is an equilibrium point of the deterministic model.
\end{proof}

The above lemma shows that the asymptotic stability of equilibrium points of the deterministic model \eqref{deterministic} can be analyzed with the aid of the potential function $U$, according to the classic Lyapunov stability criterion \cite[Chapter 30]{walter1998ordinary}.

\begin{lemma}\label{2}
	Let $\{\omega_{i_1},\cdots,\omega_{i_m}\}$ be an arbitrary basis of ${\rm span}(V(\mathcal{R}))$ and let $M = (\omega_{i_1}^T,\cdots,\omega^T_{i_m})$ be a $d\times m$ matrix. Then for any $x\in \mathbb{R}^d_{>0}$ and $y\in{\rm span}(V(\mathcal{R}))$, we have
	\begin{equation*}
	L(x,y) = \max_{\theta\in\mathbb{R}^m}\Big(\theta\cdot yA
	-\sum_{p=1}^{r}\left[f^+_p(x)\left(e^{\theta\cdot\omega_pA}-1\right)
	+f^-_p(x)\left(e^{-\theta\cdot\omega_pA}-1\right)\right]\Big),
	\end{equation*}
	where $A = M(M^TM)^{-1}$ and $L(x,y)$ is the Lagrangian defined in \eqref{localrate}. Furthermore, there exists a unique $\theta = \theta_0\in\mathbb{R}^m$ such that the maximum is attained.
\end{lemma}

\begin{proof}
	Let $\{v_1,\cdots,v_{d-m}\}$ be an arbitrary basis of ${\rm span}(V(\mathcal{R}))^\perp$ and let
	\begin{equation*}
	N = (v^T_1,\dots,v^T_{d-m})_{d\times(d-m)},\;\;\;
	B = \begin{pmatrix} (M^TM)^{-1}M^T \\ N^T \end{pmatrix}_{d\times d}
	\end{equation*}
	be two matrices. We next prove that the square matrix $B$ is invertible. To this end, we consider the system of linear equations $Bz^T = 0$. Since $Bz^T = 0$, we have $M^Tz^T=0$ and $N^Tz^T=0$. Since the columns of $M$ and $N$ constitute a basis of $\mathbb{R}^d$, we conclude that $z = 0$. Thus the system of linear equations $Bz^T = 0$ has only the zero solution, which implies that $B$ is invertible.
	
	For any $x\in\mathbb{R}^d_{>0}$ and $y\in{\rm span}(V(\mathcal{R}))$, we have
	\begin{equation}\label{substitution}
	\begin{aligned}
	L(x,y) &= \sup_{\theta\in \mathbb{R}^d}\Big(\theta\cdot y-\sum_{p=1}^{r}\left[
	f^+_p(x)(e^{\omega_p\cdot \theta}-1)+f^-_p(x)(e^{-\omega_p\cdot \theta}-1)\right]\Big)\\
	&=\sup_{\theta\in\mathbb{R}^d}\Big(\theta B\cdot y-\sum_{p=1}^{r}\left[
	f^+_p(x)(e^{\omega_p\cdot\theta B}-1)+f^-_p(x)(e^{-\omega_p\cdot\theta B}-1)\right]\Big)\\
	&=\sup_{\theta\in\mathbb{R}^d}\Big(\theta\cdot yB^T-\sum_{p=1}^{r}\left[
	f^+_p(x)(e^{\theta\cdot\omega_pB^T}-1)+f^-_p(x)(e^{-\theta\cdot\omega_pB^T}-1)\right]\Big),
	\end{aligned}
	\end{equation}
	where we have used the fact that $B$ is invertible. Since $y\in {\rm span}(V(\mathcal{R}))$, there exist $k_1,\cdots,k_m\in\mathbb{R}$ such that $y=k_1\omega_{i_1}+\dots+k_m\omega_{i_m}$. Since the columns of $M$ and the columns of $N$ are orthogonal, we have
	\begin{equation*}
	yB^T = (k_1,\dots,k_m)M^T(M(M^TM)^{-1},N) = (k_1,\dots,k_m,0,\dots,0).
	\end{equation*}
	Similarly, we can prove that the last $d-m$ components of $\omega_pB^T$ are $0$ for all $1\leq p \leq r$. Thus the supremum in the last equality of \eqref{substitution} can be taken over the first $m$ components of $\theta$, namely
	\begin{equation*}
	\begin{aligned}
	L(x,y) &=\sup_{\theta\in\mathbb{R}^m}G(x,y,\theta),
	\end{aligned}
	\end{equation*}
	where
	\begin{equation}\label{G}
	G(x,y,\theta) = \theta\cdot yA
	-\sum_{p=1}^{r}\left[f^+_p(x)\left(e^{\theta\cdot\omega_pA}-1\right)
	+f^-_p(x)\left(e^{-\theta\cdot\omega_pA}-1\right)\right],\;\;\;\theta\in\mathbb{R}^m,
	\end{equation}
	and $A = M(M^TM)^{-1}$. Direct computations show that the Hessian matrix of the function $G(x,y,\theta)$ with respect to $\theta$ is given by
	\begin{equation*}
	-A^T(\omega_{1}^T,\omega_{2}^T,\cdots,\omega^T_r)\begin{pmatrix}
	g_1(x,\theta) & 0 & \dots & 0 \\
	0 & g_2(x,\theta) & \dots & 0  \\
	& & \ddots & \\
	0 & 0 &\dots & g_r(x,\theta)
	\end{pmatrix}
	\begin{pmatrix}\omega_1 \\ \omega_2 \\ \vdots \\ \omega_r \end{pmatrix}A,
	\end{equation*}
	where
	\begin{equation*}
	g_p(x,\theta) = f^+_p(x)e^{\theta\cdot\omega_pA}+f^-_p(x)e^{-\theta\cdot\omega_p A},\;\;\;1\leq p\leq r.
	\end{equation*}
	Since $A^T(\omega_{i_1}^T,\dots,\omega^T_{i_m}) = (M^TM)^{-1}M^TM = I$ and $g_p(x,\theta)>0$ for any $x\in\mathbb{R}^d_{>0}$, it is easy to check that the Hessian matrix is negative definite. Therefore, $G(x,y,\theta)$ is a strictly concave function with respect to $\theta$ \cite[Corollary 3.8.6]{niculescu2006convex}. We next prove that
	\begin{equation}\label{infinity}
	\lim_{\|\theta\|\rightarrow\infty}\theta\cdot yA
	-\sum_{p=1}^{r}\left[f^+_p(x)\left(e^{\theta\cdot\omega_pA}-1\right)
	+f^-_p(x)\left(e^{-\theta\cdot\omega_pA}-1\right)\right] = -\infty.
	\end{equation}
	Since $yA = (k_1,\cdots,k_m)$, for any $\theta = (\theta_1,\cdots,\theta_m)$, we have
	\begin{equation*}
	\begin{aligned}
	&\;\theta\cdot yA
	-\sum_{p=1}^{r}\left[f^+_p(x)\left(e^{\theta\cdot\omega_pA}-1\right)
	+f^-_p(x)\left(e^{-\theta\cdot\omega_pA}-1\right)\right]\\
	\leq&\; \sum_{p=1}^mk_p\theta_p-
	\sum_{p=1}^m\left[f^+_{i_p}(x)e^{\theta\cdot\omega_{i_p}A}+f^-_{i_p}(x)e^{-\theta\cdot\omega_{i_p} A}\right]+\sum_{p=1}^{r}\left[f^+_p(x)+f_p^-(x)\right]\\
	=&\; \sum_{p=1}^mk_p\theta_p-
	\sum_{p=1}^m\left[f^+_{i_p}(x)e^{\theta_p}+f^-_{i_p}(x)e^{-\theta_p}\right]+\sum_{p=1}^r\left[f^+_p(x)+f_p^-(x)\right]\\
	=&\; \sum_{p=1}^m\left[k_p\theta_p-\left(f^+_{i_p}(x)e^{\theta_p}+f^-_{i_p}(x)e^{-\theta_p}\right)\right]
	+\sum_{p=1}^r\left[f^+_p(x)+f_p^-(x)\right].
	\end{aligned}
	\end{equation*}
	Therefore, \eqref{infinity} follows directly from the fact that exponential functions grow much faster than linear functions. Since $G(x,y,\theta)$ is a strictly concave function with respect to $\theta$ and since \eqref{infinity} holds, it follows that $G(x,y,\theta)$ must attain its maximum at a unique $\theta = \theta_0\in\mathbb{R}^m$ \cite[Chapter B, Theorem 4.1.1]{hiriart2012fundamentals}.
\end{proof}

For any absolutely continuous trajectory $\phi:[0,T]\rightarrow \mathbb{R}^d_{>0}$ satisfying
\begin{equation*}
\phi(0)=x_0,\;\;\;\phi(T) = y,\;\;\;\phi(\cdot)\in x_0+\mathrm{span}(V(\mathcal{R})),
\end{equation*}
we define
\begin{equation*}
S(\phi) = \int_0^TF(\phi(t))\cdot\dot{\phi}(t){\rm d}t
\end{equation*}
to be the line integral of the vector field $F$ along the trajectory $\phi$. If the system satisfies local detailed balance, then the vector field $F$ has a potential function $U$. In this case, $S(\phi)$ can be represented by the potential function $U$ as
\begin{equation*}
S(\phi) = -\int_0^T\nabla U(\phi(t))\cdot\dot{\phi}(t){\rm d}t = U(x_0)-U(y),
\end{equation*}
which only depends on the endpoints of the trajectory $\phi$ and thus is ``path-independent". Moreover, we define the {\em reversed trajectory} of $\phi$ as
\begin{equation*}
\phi^-(t)=\phi(T-t),\;\;\;0\leq t\leq T.
\end{equation*}

\begin{lemma}\label{substract}
	Suppose that a reaction network satisfies local detailed balance. Then for any absolutely continuous trajectory $\phi:[0,T]\rightarrow \mathbb{R}^d_{>0}$ satisfying
	\begin{equation*}
	\phi(0)=x_0,\;\;\;\phi(T)=y,\;\;\;\phi(\cdot)\in x_0+\mathrm{span}(V(\mathcal{R})),
	\end{equation*}
	we have
	\begin{equation}\label{eq7}
	I_{y,T}(\phi^-)-I_{x_0,T}(\phi) = U(x_0)-U(y),
	\end{equation}
	where $\phi^-$ is the reversed trajectory of $\phi$.
\end{lemma}

\begin{proof}[Proof of Lemma \ref{substract}]
	Since $\phi(\cdot)\in x_0+\mathrm{span}(V(\mathcal{R}))$, we have $\dot\phi(\cdot)\in\mathrm{span}(V(\mathcal{R}))$. It then follows from Lemma \ref{2} that for each $0\leq t\leq T$, there exists a unique $\theta = \theta_1(t)\in\mathbb{R}^m$ such that the following maximum is attained:
	\begin{equation}\label{temp1}
	L(\phi(t),\dot{\phi}(t))
	= \max_{\theta\in\mathbb{R}^m}G(\phi(t),\dot\phi(t),\theta),
	\end{equation}
	where $G(x,y,\theta)$ is the function defined in \eqref{G}. Using Lemma \ref{2} again shows that for each $0\leq t\leq T$, there exists a unique $\theta = \theta_2(t)\in\mathbb{R}^m$ such that the following maximum is attained:
	\begin{equation}\label{temp2}
	L(\phi(t),-\dot{\phi}(t))
	= \max_{\theta\in\mathbb{R}^m}G(\phi(t),-\dot\phi(t),\theta).
	\end{equation}
	Since the maximum in \eqref{temp1} is attained at $\theta = \theta_1(t)$, taking the derivatives of $G(\phi(t),\dot\phi(t),\theta)$ with respect to $\theta$ and evaluating at $\theta = \theta_1(t)$ yields
	\begin{equation*}
	\dot{\phi}(t)A = \sum_{p=1}^r\left(f^+_p(\phi(t))e^{\theta_1(t)\cdot\omega_pA}
	-f^-_p(\phi(t))e^{-\theta_1(t)\cdot\omega_pA}\right)\omega_pA,
	\end{equation*}
	where $A = M(M^TM)^{-1}$. By the proof of Lemma \ref{2}, $G(\phi(t),\dot\phi(t),\theta)$ is a strictly concave function with respect to $\theta$. This shows that $\theta = \theta_1(t)$ is the unique solution of the following equation
	\begin{equation}\label{important}
	\dot{\phi}(t)A = \sum_{p=1}^r\left(f^+_p(\phi(t))e^{\theta\cdot\omega_pA}
	-f^-_p(\phi(t))e^{-\theta\cdot\omega_pA}\right)\omega_pA.
	\end{equation}
	Similarly, since the maximum in \eqref{temp2} is attained at $\theta = \theta_2(t)$, we obtain
	\begin{equation}\label{second}
	\dot{\phi}(t)A = -\sum_{p=1}^r\left(f^+_p(\phi(t))e^{\theta_2(t)\cdot\omega_pA}
	-f^-_p(\phi(t))e^{-\theta_2(t)\cdot\omega_pA}\right)\omega_pA.
	\end{equation}
	To proceed, let
	\begin{gather*}
	f(x) = \left(\log\frac{f^+_{i_1}(x)}{f^-_{i_1}(x)},\cdots,\log\frac{f_{i_m}^+(x)}{f^-_{i_m}(x)}\right),
	\;\;\;x\in\mathbb{R}^d_{>0}.
	\end{gather*}
	It then follows from \eqref{crucial} that
	\begin{equation*}
	f(x)\cdot\omega_{p}A = F(x)\cdot\omega_{p} = \log\frac{f^+_p(x)}{f^-_p(x)}.
	\end{equation*}
	This clearly shows that
	\begin{equation}\label{relation}
	f^+_p(\phi(t)) = f^-_p(\phi(t))e^{f(\phi(t))\cdot\omega_pA},\;\;\;
	f^-_p(\phi(t)) = f^+_p(\phi(t))e^{-f(\phi(t))\cdot\omega_pA},
	\end{equation}
	which implies that
	\begin{equation*}
	\begin{split}
	&\;f^-_p(\phi(t))e^{-\theta_2(t)\cdot\omega_pA}-f^+_p(\phi(t))e^{\theta_2(t)\cdot\omega_pA}\\
	=&\;f^+_p(\phi(t))e^{-[\theta_2(t)+f(\phi(t))]\cdot\omega_pA}
	-f^-_p(\phi(t))e^{[\theta_2(t)+f(\phi(t))]\cdot\omega_pA}.
	\end{split}
	\end{equation*}
	Substituting this equation into \eqref{second}, it is easy to check that $\theta = -\theta_2(t)-f(\phi(t))$ is also a solution of the equation \eqref{important}. By the uniqueness of the solution of the equation \eqref{important}, we immediately obtain
	\begin{equation*}
	\theta_1(t)+\theta_2(t) = -f(\phi(t)).
	\end{equation*}
	It thus follows from \eqref{temp1}, \eqref{temp2}, and \eqref{relation} that
	\begin{equation*}
	\begin{split}
	&\; L(\phi(t),-\dot{\phi}(t))-L(\phi(t),\dot{\phi}(t))\\
	=&\; G(\phi(t),-\dot\phi(t),\theta_2(t))-G(\phi(t),\dot\phi(t),\theta_1(t))\\
	=&\; -(\theta_1(t)+\theta_2(t))\cdot\dot{\phi}(t)A\\
	&\; +\sum_{p=1}^r \left[f^+_p(\phi(t))e^{(\theta_1(t)+\theta_2(t))\cdot \omega_pA}-f^-_p(\phi(t))\right]\left(e^{-\theta_2(t)\cdot \omega_pA}-e^{-\theta_1(t)\cdot \omega_pA}\right)\\
	=&\; f(\phi(t))\cdot\dot{\phi}(t)A = F(\phi(t))\cdot\dot{\phi}(t).
	\end{split}
	\end{equation*}
	For the reversed trajectory $\phi^-$, note that
	\begin{equation*}
	L(\phi^-(T-t),\dot\phi^-(T-t)) = L(\phi(t),-\dot{\phi}(t)).
	\end{equation*}
	Finally, we obtain
	\begin{equation*}
	\begin{aligned}
	I_{y,T}(\phi^-)-I_{x_0,T}(\phi)
	&= \int_0^T\left[(L(\phi^-(T-t),\dot\phi^-(T-t))-L(\phi(t),\dot{\phi}(t))\right]{\rm d} t\\
	&= \int_0^T\left[L(\phi(t),-\dot{\phi}(t))-L(\phi(t),\dot{\phi}(t))\right]{\rm d}t\\
	&= \int_0^TF(\phi(t))\cdot\dot{\phi}(t){\rm d}t = S(\phi) = U(x_0)-U(y).
	\end{aligned}
	\end{equation*}
	This completes the proof.
\end{proof}

The following lemma is exactly Theorem \ref{result2}(e).

\begin{lemma}\label{le2}
	Suppose that a reaction network satisfies local detailed balance. Let $c\in\mathbb{R}^d_{>0}\cap(x_0+\mathrm{span}(V(\mathcal{R})))$ be an equilibrium point of the deterministic model \eqref{deterministic}. If $y\in\mathbb{R}^d_{>0}$ is attracted to $c$ for the deterministic model \eqref{deterministic}, then
	\begin{equation*}
	W(c,y) = U(y)-U(c).
	\end{equation*}
\end{lemma}

\begin{proof}
	Let $\phi:[0,T]\rightarrow \mathbb{R}^d_{>0}$ be an arbitrary absolutely continuous trajectory satisfying
	\begin{equation*}
	\phi(0)=c,\;\;\;\phi(T)=y,\;\;\;\phi(\cdot)\in x_0+\mathrm{span}(V(\mathcal{R})).
	\end{equation*}
	It thus follows from Lemma \ref{substract} that
	\begin{equation*}
	I_{c,T}(\phi) = U(y)-U(c)+I_{y,T}(\phi^-)\geq U(y)-U(c),
	\end{equation*}
	where $\phi^-$ is the reversed trajectory of $\phi$. Taking the infimum over $\phi$ on both sides of this equation yields
	\begin{equation}\label{geq}
	W(c,y) \geq U(y)-U(c).
	\end{equation}
	On the other hand, let $\phi_{y}(t)$ denote the trajectory of the deterministic model \eqref{deterministic} starting from $y$. In addition, let
	\begin{equation*}
	\psi_T(t)=\phi_{y}(T-t),\;\;\;0\leq t\leq T
	\end{equation*}
	denote the reversed trajectory of $\phi_{y}$ over the interval $[0,T]$. Applying Lemma \ref{substract} again shows that
	\begin{equation*}
	I_{\phi_{y}(T),T}(\psi_T) = U(y)-U(\phi_y(T))+I_{y,T}(\phi_{y}) = U(y)-U(\phi_y(T)).
	\end{equation*}
	Moreover, let
	\begin{equation*}
	\zeta_T(t) = c+\frac{(\phi_{y}(T)-c)t}{\|\phi_{y}(T)-c\|},\;\;\;0\leq t\leq\|\phi_{y}(T)-c\|,
	\end{equation*}
	be an absolutely continuous trajectory from $c$ to $\phi_{y}(T)$. Recall that for any $y\in{\rm span}(V(\mathcal{R}))$, we have
	\begin{equation*}
	\begin{aligned}
	L(x,y) =& \sup_{\theta\in \mathbb{R}^m}\Big(\theta\cdot yA-\sum_{p=1}^{r}
	\left[f^+_p(x)(e^{\theta\cdot\omega_pA}-1)+f^-_p(x)(e^{-\theta\cdot\omega_pA}-1)\right]\Big)\\
	=&\sup_{\theta\in \mathbb{R}^m}\Big(\theta\cdot yA-\sum_{p=1}^{r}
	\left[f^+_p(x)e^{\theta\cdot\omega_pA}+f^-_p(x)e^{-\theta\cdot \omega_pA}\right]\Big)
	+\sum_{p=1}^r\left[f_p^+(x)+f_p^-(x)\right],
	\end{aligned}
	\end{equation*}
	where $A = M(M^TM)^{-1}$. Since $y$ is attracted to $c$, for any $\epsilon>0$, we have $\|\phi_y(T)-c\|\leq \epsilon$ when $T$ is sufficiently large. For convenience, set
	\begin{equation*}
	\begin{aligned}
	C_0&=\min_{\|x-c\|\leq \epsilon}\{f_1^+(x),f_1^-(x),\dots,f_r^+(x),f_r^-(x)\},\\
	C_1&=\max_{\|x-c\|\leq \epsilon}\{f_1^+(x),f_1^-(x),\dots,f_r^+(x),f_r^-(x)\}.
	\end{aligned}
	\end{equation*}
	For any $\theta = (\theta_1,\cdots,\theta_m)\in\mathbb{R}^m$, whenever $\|x-c\|\leq \epsilon$, we have
	\begin{equation*}
	\begin{aligned}
	\sum_{p=1}^{r}\left(f^+_p(x)e^{\theta\cdot\omega_pA}+f^-_p (x)e^{-\theta\cdot \omega_pA}\right)
	&\geq C_0\sum_{p=1}^r\left(e^{\theta\cdot \omega_pA}+e^{-\theta\cdot \omega_pA}\right)\\
	&\geq C_0\sum_{p=1}^m\left(e^{\theta\cdot\omega_{i_p}A}+e^{-\theta\cdot \omega_{i_p}A}\right)
	= C_0\sum_{p=1}^m\left(e^{\theta_p}+e^{-\theta_p}\right).
	\end{aligned}
	\end{equation*}
	Thus for any $y = k_1\omega_{i_1}+\dots+k_m\omega_{i_m}\in\mathrm{span}(V(\mathcal{R}))$, whenever $\|x-c\|\leq \epsilon$, we have
	\begin{equation}\label{legendre}
	\begin{aligned}
	L(x,y)-\sum_{p=1}^r\left[f_p^+(x)+f_p^-(x)\right] =&\;\sup_{\theta\in \mathbb{R}^m}\Big(\theta\cdot yA-\sum_{p=1}^{r}
	\left[f^+_p (x)e^{\theta\cdot\omega_pA}+f^-_p (x)e^{-\theta\cdot\omega_pA}\right]\Big)\\
	\leq&\; \sup_{\theta\in \mathbb{R}^m}\Big(\sum_{p=1}^m k_p\theta_p-C_0\sum_{p=1}^m\left(e^{\theta_p}+e^{-\theta_p}\right)\Big)\\
	=&\; \sum_{p=1}^m \sup_{\theta_p\in \mathbb{R}}
	\left(k_p\theta_p-C_0\left(e^{\theta_p}+e^{-\theta_p}\right)\right)\\
	=&\; \sum_{p=1}^mk_p\theta_p^\star
	-C_0\left(e^{\theta_p^\star}+e^{-\theta^\star_p}\right) := f^\star(y),
	\end{aligned}
	\end{equation}
	where
	\begin{equation*}
	\theta^\star_p=\log\frac{k_p+\sqrt{k_p^2+4C_0^2}}{2C_0},\;\;\;1\leq p\leq m.
	\end{equation*}
	It is easy to see that the function $f^\star$ is continuous on $\mathrm{span}(V(\mathcal{R}))$. Since $\|\dot{\zeta_T}(t)\|\equiv 1$ and $f^\star$ is bounded on compact sets, there exists a constant $C_2>0$ such that for any $T\geq 0$ and $0\leq t\leq\|\phi_{y}(T)-c\|$,
	\begin{equation*}
	|f^\star(\dot{\zeta_T}(t))|\leq C_2.
	\end{equation*}
	Thus when $T$ is sufficiently large, it follows from \eqref{legendre} that for any $0\leq t\leq\|\phi_{y}(T)-c\|$,
	\begin{equation*}
	L(\zeta_T(t),\dot{\zeta_T}(t)) \leq \sum_{p=1}^r\left[f_p^+(\zeta_T(t))+f_p^-(\zeta_T(t))\right]+ f^\star(\dot{\zeta_T}(t)) \leq 2rC_1+C_2,
	\end{equation*}
	where we have used the fact that $\|\zeta_T(t)-c\|\leq\epsilon$ when $T$ is sufficiently large. Finally, we obtain
	\begin{equation*}
	I_{c,\|\phi_{y}(T)-c\|}(\zeta_T)
	=\int_0^{\|\phi_{y}(T)-c\|}L(\zeta_T(t),\dot{\zeta_T}(t)){\rm d}t
	\leq (2rC_1+C_2)\epsilon.
	\end{equation*}
	Combining $\zeta_T$ and $\psi_T$, we obtain an absolutely continuous trajectory from $c$ to $y$. Therefore, we have
	\begin{equation*}
	W(c,y)\leq I_{c,\|\phi_{y}(T)-c\|}(\zeta_T)+I_{\phi_{y}(T),T}(\psi_T)
	\leq U(y)-U(\phi_{y}(T))+(2rC_1+C_2)\epsilon.
	\end{equation*}
	Since $y$ is attracted to $c$, taking $T\rightarrow \infty$ in the above equation yields
	\begin{equation}\label{leq}
	W(c,y)\leq U(y)-U(c),
	\end{equation}
	where we have used the arbitrariness of $\epsilon$. Finally, the desired result follows from \eqref{geq} and \eqref{leq}.
\end{proof}

\section*{Acknowledgement}
We thank Prof. Vadim A. Malyshev and Sergey A. Pirogov for stimulating discussions via email, and thank Prof. Hao Ge and Dr. Xiao Jin for providing some useful references. We are also grateful to the anonymous referees for their valuable comments and suggestions which helped us greatly in improving the quality of this paper. C. J. acknowledges support from the NSAF grant in National Natural Science Foundation of China with grant No. U1930402. D.-Q. Jiang is supported by National Natural Science Foundation of China with grant No. 11871079.

\setlength{\bibsep}{5pt}
\small\bibliographystyle{nature}

\end{document}